\documentclass[a4paper, leqno]{article}
\usepackage[english]{babel}
\usepackage{amsmath,amscd,amssymb,amsthm}
\usepackage{physics}
\usepackage{pb-diagram}
\usepackage{graphics} 
\usepackage{bm}
\usepackage{tikz} 
\usepackage{verbatim}
\usepackage{centernot,cancel}
\newtheorem{definition}{Definition}[section]
\newtheorem{theorem}{Theorem}[section]
\newtheorem{lemma}{Lemma}[section]
\newtheorem{proposition}{Proposition}[section]

\newtheorem{remark}{Remark}[section]
\begin{document}
\title{Dimensional reduction of stable Higgs bundle and the Doubly-Coupled Vortex Equations}
\author{Takashi Ono\thanks{Department of Mathematics, Graduate School of Science, Osaka University, Osaka, Japan, u708091f@ecs.osaka-u.ac.jp}}
\date{}
\maketitle
\begin{abstract}
Let $X$ be a compact Riemann surface and $\mathbb{P}^1$ be the complex projective line.
In this paper, we introduce an equation which we call the doubly-coupled vortex equation on $X$.
We show that the existence of a solution of the doubly-coupled is equivalent to the existence of an $SU(2)$-invariant Hermitian-Einstein metric on certain Higgs bundles over $X\times \mathbb{P}^1$.
By applying the Kobayashi-Hitchin correspondence for Higgs bundles, we further show that the existence of a solution to the doubly-coupled vortex equation is equivalent to the stability of the associated Higgs quadruplet.
\end{abstract}

\section{Introduction}
Let $(X,\omega)$ be a compact Riemann Surface with a K\"ahler form $\omega$.
A \textit{holomorphic triple} over $X$ is a triple $((E_1, \overline\partial_{E_1}),(E_2,\overline\partial_{E_2}),\psi)$ consisting of holomorphic bundles $(E_1, \overline\partial_{E_1})$ and $(E_2,\overline\partial_{E_2})$ over $X$ and $\psi:E_2\to E_1$ is a holomorphic bundle morphism.
Holomorphic triples are a generalization of \textit{pairs}, which are pairs of a holomorphic bundle and a holomorphic section. 
Pairs were introduced in \cite{B1, B2} and triples were developed in \cite{G1, G2, BG}.\par
The stability of holomorphic bundles is generalized to holomorphic triples. 
There is also a Kobayashi-Hitchin type correspondence for triples: the stability of a pair corresponds to the existence of a solution to the vortex equation, and the stability of a triple corresponds to the existence of a solution to the coupled vortex equation.\par
In \cite{G2}, it was shown that solutions of the coupled vortex equation on $X$ arise via dimensional reduction of $SU(2)$-invariant Hermitian-Einstein metrics on holomorphic bundles on $X\times \mathbb{P}^1$, where $\mathbb{P}^1$ is the complex projective line. 
The classical Kobayashi-Hitchin correspondence asserts that the existence of a Hermitian-Einstein metric is equivalent to the stability of the corresponding holomorphic bundle.
 In \cite{BG}, it was shown that stable holomorphic triples on $X$ correspond to stable holomorphic vector bundles on $X \times \mathbb{P}^1$.
The ideas developed in these works were generalized in subsequent papers \cite{AG1, AG2, AG3, BGK}.\par

In this paper, we introduce an equation which we call the \textit{doubly-coupled vortex equation} on $X$.
We show that solutions to the doubly-coupled vortex equation naturally arise via dimensional reduction from $SU(2)$-invariant Hermitian-Einstein metrics on Higgs bundles over $ X\times\mathbb{P}^1$.
We also introduce a geometric object on $X$ which we call the \textit{Higgs quadruplets} and extend the notion of stability from holomorphic triples to Higgs quadruplets.
The main result of this paper is a Kobayashi-Hitchin type correspondence: the existence of a solution to the doubly-coupled vortex equation is equivalent to the stability of the corresponding Higgs quadruplet.
Our proof proceeds indirectly.
We first show that the existence of a solution of the doubly-coupled is equivalent to the existence of an $SU(2)$-invariant Hermitian-Einstein metric on a certain Higgs bundle over $X\times \mathbb{P}^1$.
We then show that the stability of the Higgs quadruplet is equivalent to the stability of the corresponding Higgs bundle over $X\times \mathbb{P}^1$ and apply the Kobayashi-Hitchin correspondence for Higgs bundles as established by Hitchin and Simpson \cite{Hit, S1}.
\par

To formulate the results rigorously, we now introduce the underlying mathematical framework for the doubly-coupled vortex equations and the notion of Higgs quadruplets.\par

Let $(E,\overline\partial_{E},\theta)$ be a Higgs bundle over $X$: $(E,\overline\partial_E)$ is a holomorphic bundle over $X$ and $\theta$ is a $\mathrm{End}E$-valued holomorphic 1-form.
Let $h$ be a hermitian metric of $E$, $\partial_h$ be the (1,0)-part of the Chern connection with respect to $\overline\partial_E$ and $h$.  
We denote by  $\theta^\dagger_h$ the formal adjoint of $\theta$ with respect to $h$.
We denote the curvature of the Chern connection $\partial_h+\overline\partial_E$ as $F_h$ and $[\theta,\theta^\dagger_h]:=\theta\wedge\theta^\dagger_h+\theta^\dagger_h\wedge \theta$.\par

We say that a quadruplet $Q=((E_1,\overline\partial_{E_1},\theta_1),(E_2,\overline\partial_{E_2},\theta_2),\phi,\psi)$ is a Higgs quadruplet if $(E_1,\overline\partial_{E_1},\theta_1)$ and $(E_2,\overline\partial_{E_2},\theta_2)$ are Higgs bundles over $X$, $\phi:E_1\to E_2,\psi:E_2\to E_1$ are Higgs bundle morphisms (i.e. $\phi\circ(\overline\partial_{E_1}+\theta_{E_1})=(\overline\partial_{E_2}+\theta_{E_2})\circ\phi, (\overline\partial_{E_1}+\theta_{E_1})\circ\psi=\psi\circ(\overline\partial_{E_2}+\theta_{E_2})$ holds), and $\phi\circ\psi=\psi\circ\phi=0$ holds.
This last condition arises naturally from the dimensional reduction  (See Section \ref{sec 3}).
Let $h_1$ and $h_2$ be hermitian metric of $E_1$ and $E_2$.
We say that the triple $(Q,h_1,h_2)$ is a solution of the doubly-coupled $\tau$-vortex equation if the following system holds: 
\begin{equation*}\left\{ \,
\begin{split}
&\Lambda (F_{h_1}+[\theta_1,\theta^\dagger_{1,h_1}])-\sqrt{-1}\phi^\ast\circ \phi-\sqrt{-1}\psi\circ \psi^\ast+2\pi\sqrt{-1}\tau\mathrm{Id}_{E_1}=0,\\
&\Lambda (F_{h_2}+[\theta_2,\theta^\dagger_{2,h_2}])+\sqrt{-1}\phi\circ \phi^\ast+\sqrt{-1}\psi^\ast\circ \psi+2\pi\sqrt{-1}\tau'\mathrm{Id}_{E_2}=0
\end{split}
\right.
\end{equation*}
holds.
Here $\phi^\ast$ and $\psi^\ast$ are the formal adjoint of $\phi$ and $\psi$ with respect to the metrics $h_1$ and $h_2$.
The constants $\tau$ and $\tau'$ cannot be independent (See Section \ref{sec 3}).
\par
We named the equation the doubly-coupled $\tau$-vortex equation to reflect its structure, where two Higgs bundles are coupled via two morphisms in both directions.
This terminology generalizes the coupled $\tau$-vortex equation introduced in \cite{G2} and studied further in \cite{BG, BGP}, which involved a single pair of bundles coupled by a single morphism.
 
\par
We now explain how a triple $(Q,h_1,h_2)$ that is a solution of the doubly-coupled $\tau$-vortex equation is related to an $SU(2)$-invariant Hermitian-Einstein metric on a certain Higgs bundle over $X\times \mathbb{P}^1$.
This perspective generalizes the dimensional reduction framework of \cite{G2} and provides a foundation for the Kobayashi-Hitchin type correspondence.
Before we proceed, we prepare some notations.
Let $\omega_{\mathbb{P}^1}$ be the Fubini-study form on $\mathbb{P}^1$ and 
 define a constant $\sigma$ by 
\begin{equation}\label{intro eq1}
\sigma:=\frac{(\mathrm{rank}E_1+\mathrm{rank}E_2)\tau-\mathrm{deg}E_1-\mathrm{deg}E_2}{\mathrm{rank}E_2}.
\end{equation} 
We assume that $\sigma>0$ holds.
We consider the K\"ahler form $\Omega_\sigma:=(\frac{\sigma}{2}\omega)\oplus \omega_{\mathbb{P}^1}$ on $X\times \mathbb{P}^1$, where
we normalize both $\omega$ and $ \omega_{\mathbb{P}^1}$ so that $\int_X\omega=\int_{\mathbb{P}^1} \omega_{\mathbb{P}^1}=1$.
\par

We assume $SU(2)$ acts on $X$ trivially and $\mathbb{P}^1$ as described in Section \ref{sec2}.
We define an $SU(2)$-invariant Higgs bundle $(F,\overline\partial_F,\theta_F)$ over $X\times \mathbb{P}^1$ by:
 \begin{align*}
&F=p^\ast E_1\oplus p^\ast E_2\otimes q^\ast \mathcal{O}_{\mathbb{P}^1}(2),\\
&\overline\partial_F:=
\begin{pmatrix}
p^\ast\overline\partial_{E_1}&p^\ast\psi\otimes q^\ast\alpha\\
0& p^\ast\overline\partial_{E_2}\otimes Id+Id\otimes q^\ast\overline\partial_{\mathcal{O}_{\mathbb{P}^1}(2)}
\end{pmatrix}
,\\
&\theta_F:=
\begin{pmatrix}
p^\ast \theta_1&0\\
p^\ast\phi\otimes q^\ast \beta&p^\ast\theta_2
\end{pmatrix}.
\end{align*}

See Proposition \ref{prop 4.5} for the proof that $(F,\overline\partial_F,\theta_F)$ is indeed a Higgs bundle over $X\times \mathbb{P}^1$ and see Section \ref{sec 2.3} for the definitions of $\alpha$ and $\beta$.

\begin{proposition}[Proposition \ref{prop 4.7}, Proposition \ref{prop 5.1}]\label{prop 1.1}
The following conditions are equivalent.
\begin{itemize}
\item There exist hermitian metrics $h_1$ and $h_2$ on $E_1$ and $E_2$ such that $(Q,h_1,h_2)$ is a solution of the doubly-coupled $\tau$-vortex equation.
\item There exist $SU(2)$-invariant Hermitian-Einstein metrics on the Higgs bundle $(F,\overline\partial_F,\theta_F)$.
\end{itemize}
\end{proposition}
 
 In Section \ref{sec 5.1}, we introduce the notion of $\sigma$-stability for Higgs quadruplets.
 This is a generalization of $\sigma$-stability for holomorphic triples defined in \cite{BG}.
 Assuming that $\tau$ and $\sigma$ are related as in equation (\ref{intro eq1}), we obtain:
 \begin{proposition}[Proposition \ref{prop 5.4}]\label{prop 1.2}
 The following conditions are equivalent.
 \begin{itemize}
\item Q is $\sigma$-stable.
\item $(F,\overline\partial_F,\theta_F)$ is stable with respect to the $SU(2)$-action and the K\"ahler form $\Omega_\sigma$.
\end{itemize}
  \end{proposition}
 
 Then, combining Proposition \ref{prop 1.1} and \ref{prop 1.2}, we obtain the main result of this paper:
 \begin{theorem}[Theorem \ref{thm 5.1}, Proposition \ref{prop 5.5}]
 There exist hermitian metrics $h_1$ and $h_2$ on $E_1$ and $E_2$ such that $(Q,h_1,h_2)$ is a solution of the doubly-coupled $\tau$-vortex equation if and only if $Q$ is $\sigma$-polystable.
Moreover,  if $Q$ is $\sigma$-stable, then the metrics $h_1$ and $h_2$ are unique up to scaling by positive constants.
  \end{theorem}
 
 The moduli space of holomorphic triples has been extensively studied in \cite{G2, BG, BGP}.
 It has important applications in the geometry of the Hitchin moduli space.\par
 
While the moduli spaces of holomorphic triples are K\"ahler, the author expects the moduli space $M_\tau$ of solutions to the doubly-coupled $\tau$-vortex equation to carry a hyperK\"ahler structure.
A brief discussion of this expectation is given in Section \ref{sec 3.1}.
The author hopes to explore the geometry of $M_\tau$ in more detail in future work.
 
 \subsubsection*{Notations}
Let $M$ be a complex manifold and $E$ be a complex vector bundle over $M$. We denote the space of smooth sections as $A(E)$. We denote the dual bundle of $E$ as $E^\vee$.
 We denote the space of smooth $i$-forms which take value in $E$ as $A^i(E)$ and smooth $(p,q)$-forms as $A^{p,q}(E)$.
 
\subsubsection*{Acknowledgement}
This work was supported by JSPS KAKENHI Grant Number JP24KJ1611.

\section{Preliminary}\label{sec2}
In this section, we collect some background materials and results for this paper. The notation of this section is used throughout the paper.
\subsection{Higgs bundle}
In this section, we collect some results from \cite{S1}.\par
Let $(M,\omega_M)$ be a compact K\"ahler manifold with the K\"ahler form $\omega_M.$
 A Higgs bundle $(E,\overline\partial_E,\theta)$ over $M$ is a triple such that (1) $(E,\overline\partial_E)$ is a holomorphic bundle over $M$, (2) $\theta\in A^{1,0}(\mathrm{End}E), \overline\partial_E \theta=0$, and $\theta\wedge\theta=0$. 
The $\theta$ is called the Higgs field.\par
Let $h$ be a hermitian metric of $E$. We say a connection $D$ of $E$ is a metric connection if $dh(\cdot,\cdot)=h(D\cdot,\cdot)+h(\cdot,D\cdot)$ holds. 
There exists an unique unitary connection $\nabla_h$ such that the $(0,1)$-part of $\nabla$ satisfies $\nabla_h^{0,1}=\overline\partial_E$. 
The connection is called the \textit{Chern connection} and we denote the (1,0)-part of the Chern connection as  $\nabla_h$. 
Let $\theta^\dagger_h$ be the formal adjoint of $\theta$ with respect to $h$ (i.e. $h(\theta\cdot,\cdot)=h(\cdot,\theta^\dagger_h\cdot)$ holds).\par
Let $F_{h}$ be the curvature of $\nabla_h$, $\Lambda_\omega$ be the contraction with respect to $\omega$, and $[\theta,\theta^\dagger_h]=\theta\wedge \theta^\dagger_h+\theta^\dagger_h\wedge \theta$. 
 
We say that a hermitian metric $h$ is a \textit{Hermitian-Einstein metric} if 
\begin{equation*}
\Lambda_\omega(F_h+[\theta,\theta^\dagger_h])=\lambda Id_E.
\end{equation*}
holds for some constant $\lambda$.
 We say that a pair of a Higgs bundle and a hermitian metric $(E,\overline\partial_E,\theta,h)$ is \textit{Hermitian-Einstein} if $h$ is a Hermitian-Einstein metric for $(E,\overline\partial_E,\theta)$.\par
The existence of Hermitian-Einstein metrics on a Higgs bundle $(E,\overline\partial_E,\theta)$ is related to the stability of the Higgs bundle.\par
Let $c_1(E)$ be the first Chern class of $E$.
The degree of $E$ is defined as 
\begin{equation*}
\mathrm{deg}E:=\int_Mc_1(E)\wedge \omega^{\mathrm{dim}M-1}_M.
\end{equation*}
The slope $\mu(E)$ of $E$ is defined as a ratio
\begin{equation*}
\mu(E):=\frac{\mathrm{deg}E}{\mathrm{rank}E}.
\end{equation*}
We say that a Higgs bundle $(E,\overline\partial_E,\theta)$ is \textit{stable} if for every saturated sub Higgs sheaf $\mathcal{V}$
\begin{equation*}
\frac{\mathrm{deg}\mathcal{V}}{\mathrm{rank}\mathcal{V}}<\frac{\mathrm{deg}E}{\mathrm{rank}E}
\end{equation*}
holds. 

The degree of $\mathcal{V}$ can be defined by the Cheil-Weil formula (See \cite[Lemma 3.1]{S1}). 
We say that the Higgs bundle $(E,\overline\partial_E,\theta)$ is \textit{polystable}  if we have a decomposition $(E,\overline\partial_E,\theta)=\bigoplus_i(E_i,\overline\partial_{E_i},\theta_i)$ such that each $(E_i,\overline\partial_{E_i},\theta_i)$ is stable and $\mu(E)=\mu(E_i)$ holds for every $i$.\par
Simpson proved that a Higgs bundle admits Hermitian-Einstein metrics if and only if it is polystable. He also proved it with a group action, which is crucial for this work.
 We give a review of this.\par
Let $G$ be a group.
 We assume that $G$ acts on $M$ as holomorphic automorphisms preserving the metric $\omega$ and acting compatibly by automorphism $\alpha:E\to E$ which preserves the metric $h$ and acts on $\theta$ by homotheties.\par
We say that $(E,\overline\partial_E,\theta)$ is stable with respect to the $G$-action if $(E,\overline\partial_E,\theta)$ is stable but for every $G$-invariant proper saturated sub-Higgs sheaf $\mathcal{V}\subset E$.
 Polystability with respect to the $G$-action is defined analogously.
 Then, Simpson proved the following
\begin{theorem}[\cite{S1}]\label{thm 2.1}
Let $M$ be a compact K\"ahler manifold and $(E,\overline\partial_E,\theta)$ be a Higgs bundle over $M$.
 Let $G$ be a group that acts as above. 
 Then $(E,\overline\partial_E,\theta)$ admits $G$-invariant Hermitian-Einstein metrics if and only if it is polystable with respect to the $G$-action. 
 If the Higgs bundle is stable with respect to the $G$-action, then it is unique up to positive constants.
 \end{theorem}
 \begin{remark}
 In the paper \cite{S1}, stable with respect to the $G$-action is just stated as stable. In this paper, we say stable with respect to the $G$-action to emphasize the action of $G$.
  \end{remark}
  
\subsection{Hitchin Equation}
In this section, we review some results of \cite{Hit}.\par
Let $(X,\omega_X)$ be a compact Riemann surface.
 Let $E$ be a smooth complex vector bundle on $X$ and $h$ be a hermitian metric of $E$.\par
Let $\nabla_h$ be an unitary connection of $E$ with respect to $h$ and $\Phi$ be a $\mathrm{End}E$-valued skew-symmetric 1-form with respect to $h$ (i.e. $h(\Phi\cdot,\cdot)+h(\cdot,\Phi\cdot)=0$). 
Let $\ast$ be the Hodge star.
We say that $(\nabla_h,\Phi)$ is a solution of  the \textit{Hitchin equation} if
\begin{align*}
F_{h}-\Phi\wedge\Phi&=\lambda\omega_X\mathrm{Id}_E,\\
\nabla_h\Phi&=0,\\
\nabla_h\ast\Phi&=0.
\end{align*}

holds for some constant $\lambda$.

This equation arises from the two-dimensional reduction of instantons. The equation is deeply studied in \cite{Hit}. \par
Suppose a pair $(\nabla_h,\Phi)$ satisfies the Hitchin equation. 
Set $\overline\partial_E:=\nabla^{0,1}_h, \theta:=\sqrt{-1}\Phi^{1,0}$.
Then $(E,\overline\partial_E,\theta)$ is a Higgs bundle with a Hermitian-Einstein metric $h$.
Hence, the induced Higgs bundle $(E,\overline\partial_E,\theta)$ is polystable.\par
Conversely, let $(E,\overline\partial_E,\theta)$ be a polystable Higgs bundle over $X$.
Let $h$ be a Hermitian-Einstein metric.
Such a metric exists because of polystability.
Let $\nabla_h$ be the Chern connection with respect to $\overline\partial_E$ and $h$ and $\Phi:=-\sqrt{-1}\theta+\sqrt{-1}\theta^\dagger_h$.
 Then $(\nabla_h,\Phi)$ satisfies the Hitchin equation.
Hence, we can identify a pair of a metric connection and a skew-symmetric 1-form that is a solution of the Hitchin equation and a polystable Higgs bundle.
This correspondence makes the geometry of the moduli space of the Hitchin equation very rich. 
The moduli space is called the \textit{Hitchin moduli}.\par
We give a very brief review of the construction of the Hitchin moduli. 
Although there are several ways to construct this space, we review the hyperK\"ahelr reduction way.
We omit the details of the analytic aspects.
 \par
First we define 
\begin{align*}
A(\mathfrak{u}(E))&:=\bigg\{f\in A(\mathrm{End}E)\,|\, h(f\cdot,\cdot)+h(\cdot,f\cdot)=0\bigg\},\\
A^i(\mathfrak{u}(E))&:=\bigg\{\Phi\in A^i(\mathrm{End}E)\,|\, h(\Phi\cdot,\cdot)+h(\cdot,\Phi\cdot)=0\bigg\},\\
\mathcal{G}:=A(U(E))&:=\bigg\{g\in A(GL(E))\,|\, h(g\cdot,g\cdot)=h(\cdot,\cdot)\bigg\}.
\end{align*}
Let $\mathcal{A}_h$ be the space of unitary connections. 
This space is an affine space modeled on $A^1(\mathfrak{u}(E))$. $\mathcal{G}$ acts on $\mathcal{A}_h\times A^{1}(\mathfrak{u}(E))$ as 
\begin{equation*}
\begin{array}{rccc}
& \mathcal{G}\times \mathcal{A}_h\times A^{1}(\mathfrak{u}(E))  &\longrightarrow&  \mathcal{A}_h\times A^{1}(\mathfrak{u}(E))                     \\
        & \rotatebox{90}{$\in$}&               & \rotatebox{90}{$\in$} \\
        & (g,\nabla_h,\Phi)                    & \longmapsto   & (g^{-1}\nabla_hg,g^{-1}\Phi g).
\end{array}
\end{equation*}
Let $(\nabla_h,\Phi)\in\mathcal{A}_h\times A^{1}(\mathfrak{u}(E)).$ 
Then the tangent space of $\mathcal{A}_h\times A^{1}(\mathfrak{u}(E))$ at $(\nabla_h,\Phi)$ is 
\begin{equation*}
T_{(\nabla_h,\Phi)}(\mathcal{A}_h\times A^{1}(\mathfrak{u}(E)))=A^{1}(\mathfrak{u}(E))\times A^{1}(\mathfrak{u}(E)).
\end{equation*}
Let $(A_1,A_2), (B_1,B_2)\in T_{(\nabla_h,\Phi)}(\mathcal{A}_h\times A^{1}(\mathfrak{u}(E)))$. We define the symmetric form $g$ on $T_{(\nabla_h,\Phi)}(\mathcal{A}_h\times A^{1}(\mathfrak{u}(E)))$ as 
\begin{equation*}
g((A_1,A_2),(B_1,B_2)):=-\int_X\mathrm{Tr}(A_1\wedge\ast B_1+A_2\wedge\ast B_2).
\end{equation*}
We aslo define three $\mathbb{R}$-linear maps $I,J,K$ on $ T_{(\nabla_h,\Phi)}(\mathcal{A}_h\times A^{1}(\mathfrak{u}(E)))$ as 
\begin{align*}
I(A_1,A_2)&:=(\ast A_1,-\ast A_2)\\
J(A_1,A_2)&:=(-A_2,A_1)\\
K&:=IJ.
\end{align*}
It is easy to check that $I,J,K$ satisfy the quaternoic relation. 
In particular $(\mathcal{A}_h\times A^{1}(\mathfrak{u}(E)),g, I, J, K)$ is a hyperK\"ahler manifold. 
The Lie algebra of the Lie group $\mathcal{G}$ is $A(\mathfrak{u}(E))$. We regard $A^2(\mathfrak{u}(E))$ as the dual of $A(\mathfrak{u}(E))$.
Next, we define the a map
\begin{equation*}
\mu: \mathcal{A}_h\times A^{1}(\mathfrak{u}(E))\to A^{2}(\mathfrak{u}(E))^{\oplus 3}
\end{equation*} 
as
\begin{equation*}
\mu(\nabla_h,\Phi):=(F_h-\Phi\wedge\Phi, \nabla_h\Phi,\nabla_h\ast\Phi).
\end{equation*}
Then $\mu$ is actually a hyperK\"ahler moment map and the Hitchin moduli $\mathcal{M}_{\mathrm{Hit}}$ is the hyperK\"ahler quotient of this moment map i.e. $\mathcal{M}_{\mathrm{Hit}}=\mu^{-1}(\lambda\omega_X\mathrm{Id}_E,0,0)/\mathcal{G}$. 
The smooth part of this space is a finite dimension hyperK\"ahler manifold.

\subsection{$SU(2)$-action on $\mathbb{P}^1$ and $\mathcal{O}_{\mathbb{P}^1}(n)$}\label{sec 2.3}
In this section, we review the $SU(2)$-action on $\mathbb{P}^1$ and $\mathcal{O}_{\mathbb{P}^1}(n)$.\par
Let $M_2(\mathbb{C})$ be the set of $2\times 2$ complex valued matrixes.
 Recall that 
\begin{equation*}
SU(2)=\bigg\{
\begin{pmatrix}
a&-\overline{b}\\
b&\overline{a}
\end{pmatrix}\in M_2(\mathbb{C})
\,\,
\bigg|
\,\,
 |a|^2+|b|^2=1\bigg\}.
\end{equation*}
$SU(2)$ acts on $\mathbb{P}^1$ as 
\begin{equation*}
\begin{array}{rccc}
& SU(2)\times\mathbb{P}^1  &\longrightarrow& \mathbb{P}^1                     \\
        & \rotatebox{90}{$\in$}&               & \rotatebox{90}{$\in$} \\
        &\bigg( \begin{pmatrix}
a&-\overline{b}\\
b&\overline{a}
\end{pmatrix}   , [z_0:z_1]   \bigg)              & \longmapsto   &  [az_0-\overline{b}z_1:bz_0+\overline{a}z_1].
\end{array}
\end{equation*}
It is well known that this action is transitive and we can regard  $\mathbb{P}^1$ as a homogeneous space $SU(2)/U(1)$.\par
We now regard $\mathbb{P}^1$ as a two-copy of complex planes $\mathbb{C}_z$ and $\mathbb{C}_w$ glued with the relation $w=1/z$.
Then the action of $g= \begin{pmatrix}
a&-\overline{b}\\
b&\overline{a}
\end{pmatrix}$
on $\mathbb{P}^1$ can be realized as 
$z\mapsto\frac{az-\overline{b}}{bz+\overline{a}}$.\par
Let $\mathbb{C}_z\times\mathbb{C}\to\mathbb{C}_z$ and $\mathbb{C}_w\times\mathbb{C}\to\mathbb{C}_w$ be the trivial line bundle on $\mathbb{C}_z$ and $\mathbb{C}_w$.
Let $e_z$ and $e_w$ be the canonical global sections of them. 
The line bundle $\mathcal{O}_{\mathbb{P}^1}(-1)$ is obtained by gluing the two trivial line bundles by the relation $e_w=\frac{e_z}{z}$.
$\mathcal{O}_{\mathbb{P}^1}(-1)$ is called the taotological bundle.\par
It is a standard fact that $\mathcal{O}_{\mathbb{P}^1}(-1)$ is a sub-bundle of the rank 2 trivial bundle $\mathbb{P}^1\times \mathbb{C}^2\to\mathbb{P}^1$. 
The standard action of $SU(2)$ on $\mathbb{C}^2$ induces an $SU(2)$-action on $\mathcal{O}_{\mathbb{P}^1}(-1)$ which is compatible with the $SU(2)$-action on $\mathbb{P}^1$. 
We write down the action explicitly.  
For each $g= \begin{pmatrix}
a&-\overline{b}\\
b&\overline{a}
\end{pmatrix}$,
it induces a map $\Psi_g:\mathcal{O}_{\mathbb{P}^1}(-1)\to \mathcal{O}_{\mathbb{P}^1}(-1)$ such that 
\begin{equation*}
\Psi_g(e_z)=(bz+\overline{a})e_{g\cdot z}.
\end{equation*}
We define a hermitian metric $h^{(-1)}$ on $\mathcal{O}_{\mathbb{P}^1}(-1)$ such that 
\begin{equation*}
h^{(-1)}(e_z,e_z):=1+|z|^2, h^{(-1)}(e_w,e_w):=1+|w|^2.
\end{equation*}
This metric is $SU(2)$-invariant: $h^{(-1)}(\Psi_g\cdot,\Psi_g\cdot)=h^{(-1)}(\cdot,\cdot)$ holds for every for every $g\in SU(2)$.\par
We define $\mathcal{O}_{\mathbb{P}^1}(1)$ to be the dual bundle of $\mathcal{O}_{\mathbb{P}^1}(-1)$. 
For every $n\in\mathbb{Z}_{\geq 0}$, we define $\mathcal{O}_{\mathbb{P}^1}(n):=\mathcal{O}_{\mathbb{P}^1}(1)^{\otimes n}$ and $\mathcal{O}_{\mathbb{P}^1}(-n):=\mathcal{O}_{\mathbb{P}^1}(-1)^{\otimes n}$. 
Each $\mathcal{O}_{\mathbb{P}^1}(n)\, (n\in\mathbb{Z})$ has an $SU(2)$-action which is induced by the $SU(2)$-action on $\mathcal{O}_{\mathbb{P}^1}(-1)$.\par
The metric $h^{(-1)}$ defines a hermitian metric $h^{(n)}$ on $\mathcal{O}_{\mathbb{P}^1}(n)$. 
 The metric $h^{(n)}$ is also $SU(2)$-invariant.
 We note that $h^{(2)}$ is the Fubini-Study metric.\par
 For later use, we recall the $SU(2)$-invariant differential form valued on $\mathcal{O}_{\mathbb{P}^1}(-2)$ and $\mathcal{O}_{\mathbb{P}^1}(2)$.
  Before we proceed, we prepare some notations. 
  Let $T^{\ast(1,0)}\mathbb{P}^1$ be the holomorphic cotangent bundle on $\mathbb{P}^1$. Note that $\mathcal{O}_{\mathbb{P}^1}(-2)\simeq T^{\ast,(1,0)}\mathbb{P}^1$ as holomorphic bundles.
  Let $T^{\ast(0,1)}\mathbb{P}$ be the anti-holomorphic cotangent bundle.
  
  \begin{lemma}\label{lem 2.1}
  \begin{equation*}
  T^{\ast(0,1)}\mathbb{P}\simeq \mathcal{O}_{\mathbb{P}^1}(2)
    \end{equation*}
    as smooth bundles.
  \end{lemma}
  \begin{proof}
  Although this is well known, we give a proof for convenience.\par
 Let $D$ be a connection of $T^{\ast(1,0)}\mathbb{P}^1$ and $F_D$ be the curvature of $D$. Then
\begin{equation*}
\mathrm{deg}(T^{\ast(1,0)}\mathbb{P}^1)=\frac{\sqrt{-1}}{2\pi}\int_{\mathbb{P}^1}\mathrm{Tr}\Lambda F_D.
\end{equation*}
This is independent of $D$ (See \cite{W}, for example). 
Since $T^{\ast(0,1)}\mathbb{P}$ is the complex conjugate of $T^{\ast(1,0)}\mathbb{P}^1$,
the connection $D$ induces a connection $\overline{D}$ on $T^{\ast(0,1)}\mathbb{P}$. 
The curvatures of $D$ and $\overline{D}$ are related as $F_{\overline{D}}=\overline{F_D}$.
 Hence
 \begin{align*}
 \mathrm{deg}(T^{\ast(0,1)}\mathbb{P})&=\frac{\sqrt{-1}}{2\pi}\int_{\mathbb{P}^1}\mathrm{Tr}\Lambda F_{\overline{D}}\\
 &=\frac{\sqrt{-1}}{2\pi}\int_{\mathbb{P}^1}\mathrm{Tr}\Lambda\overline{F_D}\\
 &=-\overline{\frac{\sqrt{-1}}{2\pi}\int_{\mathbb{P}^1}\mathrm{Tr}\Lambda F_D}\\
 &=-\overline{\mathrm{deg}(T^{\ast(1,0)}\mathbb{P}^1)}.
     \end{align*}
Since $\mathrm{deg}(\mathcal{O}_{\mathbb{P}^1}(-1))=-1$ and $\mathcal{O}_{\mathbb{P}^1}(-2)\simeq T^{\ast(1,0)}\mathbb{P}^1$, $\mathrm{deg}(T^{\ast(1,0)}\mathbb{P}^1)=-2$ and hence $\mathrm{deg}(T^{\ast(0,1)}\mathbb{P})=2$.
Recall that a smooth complex line bundle over $\mathbb{P}^1$ is uniquely determined by its degree. 
Since $\mathrm{deg}(\mathcal{O}_{\mathbb{P}^1}(2))=2$, $T^{\ast(0,1)}\mathbb{P}\simeq \mathcal{O}_{\mathbb{P}^1}(2)$ as smooth bundles.
  \end{proof}
We denote the space of smooth $SU(2)$-invariant section of $\mathcal{O}_{\mathbb{P}^1}(n)$ as $A(\mathcal{O}_{\mathbb{P}^1}(n))^{SU(2)}$, smooth $SU(2)$-invariant $\mathcal{O}_{\mathbb{P}^1}(n)$-valued 1-form as $A^1(\mathcal{O}_{\mathbb{P}^1}(n))^{SU(2)}$.

\begin{proposition}\label{prop 2.1}
(1) For $n\neq 0$, $A(\mathcal{O}_{\mathbb{P}^1}(n))^{SU(2)}$ only contains zero section. 
\par
(2) $\mathrm{dim}A^1(\mathcal{O}_{\mathbb{P}^1}(-2))^{SU(2)}=\mathrm{dim}A^1(\mathcal{O}_{\mathbb{P}^1}(2))^{SU(2)}=1$. 
In particular, let $\alpha$ and $\beta$ be the basis of $A^1(\mathcal{O}_{\mathbb{P}^1}(-2))^{SU(2)}$ and $A^1(\mathcal{O}_{\mathbb{P}^1}(2))^{SU(2)}$. 
Then $\alpha\in A^{0,1}(\mathcal{O}_{\mathbb{P}^1}(-2))$ and $\beta\in A^{1,0}(\mathcal{O}_{\mathbb{P}^1}(2)).$
\end{proposition}
\begin{proof}
This is also well known, but we give a proof for convenience.\par
(1) Suppose we have a non-zero section $a\in A(\mathcal{O}_{\mathbb{P}^1}(n))^{SU(2)}$. 
Since $a$ is $SU(2)$-invariant and $SU(2)$ acts transitive on $\mathbb{P}^1$, $a$ is everywhere non-zero.
However, this is a contradiction because if such $a$ exists, then $a$ trivialize $\mathcal{O}_{\mathbb{P}^1}(n)$ but this contradicts to $\mathrm{deg}(\mathcal{O}_{\mathbb{P}^1}(n))=n$.\par
(2) First, we have 
\begin{equation}\label{prop 2.1 eq 1}
\begin{split}
A^1(\mathcal{O}_{\mathbb{P}^1}(-2))&=A^{1,0}(\mathcal{O}_{\mathbb{P}^1}(-2))\oplus A^{0,1}(\mathcal{O}_{\mathbb{P}^1}(-2))\\
&=A(T^{\ast(1,0)}\mathbb{P}^1\otimes\mathcal{O}_{\mathbb{P}^1}(-2))\oplus A(T^{\ast(0,1)}\mathbb{P}\otimes \mathcal{O}_{\mathbb{P}^1}(-2))\\
&=A(\mathcal{O}_{\mathbb{P}^1}(-2)\otimes\mathcal{O}_{\mathbb{P}^1}(-2))\oplus A(\mathcal{O}_{\mathbb{P}^1}(2)\otimes \mathcal{O}_{\mathbb{P}^1}(-2))\\
&=A(\mathcal{O}_{\mathbb{P}^1}(-4))\oplus C^{\infty}(\mathbb{P}^1).
\end{split}
\end{equation}
Here $C^{\infty}(\mathbb{P}^1)$ is the space of the complex valued smooth functions on $\mathbb{P}^1$.
Let $C^{\infty}(\mathbb{P}^1)^{SU(2)}$ be the space of $SU(2)$-invariant smooth functions. 
Then
\begin{align*}
A^1(\mathcal{O}_{\mathbb{P}^1}(-2))^{SU(2)}&=A(\mathcal{O}_{\mathbb{P}^1}(-4))^{SU(2)}\oplus C^{\infty}(\mathbb{P}^1)^{SU(2)}\\
&=C^{\infty}(\mathbb{P}^1)^{SU(2)}.
\end{align*}
The second equation follows from $(1)$.
Since $SU(2)$ acts transitive on $\mathbb{P}^1$, $C^{\infty}(\mathbb{P}^1)^{SU(2)}=\mathbb{C}$ and therefore $\mathrm{dim} A^1(\mathcal{O}_{\mathbb{P}^1}(-2))^{SU(2)}=1$. 
By using a similar argument, we can also prove $A^1(\mathcal{O}_{\mathbb{P}^1}(-2))^{SU(2)}=C^{\infty}(\mathbb{P}^1)^{SU(2)}$ and hence $\mathrm{dim} A^1(\mathcal{O}_{\mathbb{P}^1}(2))^{SU(2)}=1$.
By the equation (\ref{prop 2.1 eq 1}), we have 
\begin{equation*}
A^1(\mathcal{O}_{\mathbb{P}^1}(-2))^{SU(2)}=A(T^{\ast(0,1)}\mathbb{P}\otimes \mathcal{O}_{\mathbb{P}^1}(-2))^{SU(2)}=A^{0,1}(\mathcal{O}_{\mathbb{P}^1}(-2))^{SU(2)}.
\end{equation*}
We also have
\begin{equation*}
A^1(\mathcal{O}_{\mathbb{P}^1}(2))^{SU(2)}=A(T^{\ast(1,0)}\mathbb{P}\otimes \mathcal{O}_{\mathbb{P}^1}(2))^{SU(2)}=A^{1,0}(\mathcal{O}_{\mathbb{P}^1}(2))^{SU(2)}.
\end{equation*}
Hence, we proved (2).\par
We give explicit realizations $\alpha$ and $\beta$. 
Let $\mathbb{C}_z\times\mathbb{C}\to\mathbb{C}_z$ and $\mathbb{C}_w\times\mathbb{C}\to\mathbb{C}_w$ be the trivial line bundle on $\mathbb{C}_z$ and $\mathbb{C}_w$.
Let $e_{n,z}$ and $e_{n_w}$ be the canonical global section.
$\mathcal{O}_{\mathbb{P}^1}(n)$ can be constructed by gluing the two trivial line bundles by $e_{n,w}=z^ne_{n,z}$. 
$\alpha$ and $\beta$ are realized as follows:
\begin{align*}
&\alpha|_{\mathbb{C}_z}=\frac{1}{(1+|z|^2)^2}d\overline{z}\otimes e_{-2,z},\,\,\alpha|_{\mathbb{C}_w}=-\frac{1}{(1+|w|^2)^2}d\overline{w}\otimes e_{-2,w}.\\
&\beta|_{\mathbb{C}_z}=dz\otimes e_{2,z},\,\,\beta|_{\mathbb{C}_w}=-dw\otimes e_{2,w}.
\end{align*}
Then $\alpha$ and $\beta$ are indeed bases of $A^1(\mathcal{O}_{\mathbb{P}^1}(-2))^{SU(2)}$ and $A^1(\mathcal{O}_{\mathbb{P}^1}(2))^{SU(2)}$. This follows from $\alpha$ and $\beta$ are everywhere non-zero and $SU(2)$-invariant.
\end{proof}

\section{The Doubly-Coupled $\tau$-Vortex Equation}\label{sec 3}
We are now able to introduce the equation which we call the \textit{doubly-coupled $\tau$-vortex equation}. 
We freely use the notation we introduced in  Section \ref{sec2}.\par
Let $(X,\omega)$ be a compact Riemann surface. We will fix $(X,\omega)$ throughout the section. We also normalize $\omega$ so that $\int_X\omega=1$.\par
 Let $E_1$ and $E_2$ be smooth complex vector bundles of $\mathrm{rank}E_i=r_i (i=1,2)$ over $X$.
 Let $h_1$ and $h_2$ be hermitian metrics on $E_1$ and $E_2$.
 Recall that $\mathcal{A}_{h_i}$ is the space of unitary connections on $(E_i,h_i)$ and $A^1(\mathfrak{u}(E_i))$ is the space of skew-symmetric $\mathrm{End}E_i$-valued 1-forms of $(E_i,h_i)$.
For $f\in A(\mathrm{Hom}(E_1,E_2))$ we denote the adjiont  of $f$ with respect to the metrics $h_1,h_2$ as $f^\ast\in A(\mathrm{Hom}(E_2,E_1))$: $f^\ast$ is the unique section that satsifies $h_2(f_1\cdot,\cdot)=h_1(\cdot,f^\ast\cdot)$.\par
We set 
\begin{equation*}
M:=\mathcal{A}_{h_1}\times A^1(\mathfrak{u}(E_1))\times\mathcal{A}_{h_2}\times A^1(\mathfrak{u}(E_2))\times  A(\mathrm{Hom}(E_1,E_2))\times A(\mathrm{Hom}(E_2,E_1)).
\end{equation*}
To connect with the holomorphic object, we define a subset $N\subset M$ as 
\begin{align*}
N=
\left\{
(\nabla_{h_1}, \Phi_1, \nabla_{h_2}, \Phi_2, \phi, \psi) \in M \ \middle| \
\begin{aligned}
&1.\quad (\nabla^{0,1}_i + \sqrt{-1}\Phi^{1,0}_i)^2 = 0 \quad (i = 1,2), \\
&2.\quad \phi \circ (\nabla^{0,1}_1 + \Phi^{1,0}_1) = (\nabla^{0,1}_2 + \Phi^{1,0}_2) \circ \phi, \\
&3.\quad (\nabla^{0,1}_1 + \Phi^{1,0}_1) \circ \psi = \psi \circ (\nabla^{0,1}_2 + \Phi^{1,0}_2), \\
&4.\quad \phi \circ \psi = \psi \circ \phi = 0
\end{aligned}
\right\}
\end{align*}

Let $(\nabla_{h_1},\Phi_1,\nabla_{h_2},\Phi_2, \phi,\psi)\in N$. 
It is clear from the definition that both $(\nabla_{h_i},\Phi_i)$ define Higgs bundles $(E_i,\nabla^{0,1}_{h_i},\sqrt{-1}\Phi^{1,0}_i)$ over $X$.
It is also clear from the definition that $\phi$ and $\psi$ define Higgs bundle morphisms $\phi:(E_1,\nabla^{0,1}_{h_1},\sqrt{-1}\Phi^{1,0}_1)\to(E_2,\nabla^{0,1}_{h_2},\sqrt{-1}\Phi^{1,0}_2)$ and $\psi:(E_1,\nabla^{0,1}_{h_2},\sqrt{-1}\Phi^{1,0}_2)\to(E_1,\nabla^{0,1}_{h_1},\sqrt{-1}\Phi^{1,0}_1)$.
The condition $\phi\circ\psi=\psi\circ\phi=0$ naturally arises from the dimensional reduction which we study in Section \ref{sec 3}. 
\par
The doubly-coupled $\tau$-vortex equation is an equation defined for elements of $N$.  

 \begin{definition}
 We say that an element 
 \begin{equation*}
(\nabla_{h_1},\Phi_1,\nabla_{h_2},\Phi_2, \phi,\psi)\in N.
\end{equation*}
is a solution of  the doubly-coupled $\tau$-vortex equation if the following equations hold:
\begin{equation}\left \{\,
\begin{split}
&\Lambda (F_{h_1}-\Phi_1\wedge\Phi_1)-\sqrt{-1}\phi^\ast\circ \phi-\sqrt{-1}\psi\circ \psi^\ast+2\pi\sqrt{-1}\tau\mathrm{Id}_{E_1}=0,\\
&\Lambda (F_{h_2}-\Phi_2\wedge\Phi_2)+\sqrt{-1}\phi\circ \phi^\ast+\sqrt{-1}\psi^\ast\circ \psi+2\pi\sqrt{-1}\tau'\mathrm{Id}_{E_2}=0
\end{split}
\right.
\end{equation}
 \end{definition}
 
Note that if we set $\Phi_1=\Phi_2=0$ and $\phi=0$, then we recover the coupled vortex equation, which is deeply studied by the series of work \cite{BG, G1, G2}.\par
Following page 11 of \cite{G2}, $\tau$ and $\tau'$ cannot be independent.
Since the trace of $\Phi_1\wedge\Phi_1,\Phi_2\wedge\Phi_2,\phi\circ\phi^\ast+\phi^\ast\circ\phi$, and $\psi\circ\psi^\ast=\psi^\ast\circ\psi$ all vanish, we have 

 \begin{equation}
\mathrm{tr}\Lambda (F_{h_1})+\mathrm{tr}\Lambda (F_{h_2})+2\pi\sqrt{-1}\tau\cdot r_1-2\pi\sqrt{-1}\tau'\cdot r_2=0.
  \end{equation}
  
  Therefore  

\begin{equation}
\tau'=-\frac{1}{r_2}\cdot(r_1\cdot\tau-\mathrm{deg}E_1-\mathrm{deg}E_2)
\end{equation}
holds.\par

As we explained in the introduction, we named the equation the doubly-coupled $\tau$-vortex equation to reflect its structure, where two Higgs bundles are coupled via two morphisms in both directions.
This terminology generalizes the coupled vortex equation introduced in \cite{G2} and studied further in \cite{BG, BGP}, which involved a single pair of bundles coupled by a single morphism.
 
 \subsection{Moment Map Picture}\label{sec 3.1}
We derive the doubly-coupled vortex equation as moment map equations.
 We ignore the analytic aspects, such as Sobolev spaces, and only give a formal picture.\par
We let the Lie group 
\begin{equation*}
\mathcal{G}_1\times \mathcal{G}_2:=A(U(E_1))\times A(U_2(E))
\end{equation*} 
acts on 
\begin{equation*}
M=\mathcal{A}_{h_1}\times A^1(\mathfrak{u}(E_1))\times\mathcal{A}_{h_2}\times A^1(\mathfrak{u}(E_2))\times  A(\mathrm{Hom}(E_1,E_2))\times A(\mathrm{Hom}(E_2,E_1)).
\end{equation*}
as 

\begin{equation*}
\begin{array}{rccc}
& \mathcal{G}_1\times \mathcal{G}_2\times M &\longrightarrow& M                   \\
        & \rotatebox{90}{$\in$}&               & \rotatebox{90}{$\in$} \\
        & (g_1,g_2,\nabla_{h_1},\Phi_1,\nabla_{h_2},\Phi_2,\phi,\psi)                    & \longmapsto   & (g_1\nabla_{h_1}g^{-1}_1,g_1\Phi_1 g^{-1}_1,g_2\nabla_{h_2}g^{-1}_2,g_2\Phi_2 g^{-1}_2, g_2\phi g^{-1}_1, g_1\psi g_2^{-1}).
\end{array}
\end{equation*}

We now define symplectic structers on $M$.
Let $x\in M$.
Let $T_xM$ be the tangent space of $M$ at $x$. Since $M$ is an affine space, we have 
\begin{align*}
T_xM=A^1(\mathfrak{u}(E_1))\times A^1(\mathfrak{u}(E_1))\times A^1(\mathfrak{u}(E_2))\times A^1(\mathfrak{u}(E_2))\times  A(\mathrm{Hom}(E_1,E_2))\times A(\mathrm{Hom}(E_2,E_1)).
\end{align*}

 Let 
 \begin{equation*}
 a=(\dot{A}_1,\dot{\Phi}_1,\dot{B}_1,\dot{\Psi}_1,f_1,g_1), \,\,b=(\dot{A}_2,\dot{\Phi}_2,\dot{B}_2,\dot{\Psi}_2,f_2,g_2)\in T_xM.
 \end{equation*}
Then we define a non-degenerate bilinear map $g_x:T_xM\times T_xM\to\mathbb{R}$ as 
\begin{align*}
g_x(a,b)=&g\bigg((\dot{A}_1,\dot{\Phi}_1,\dot{B}_1,\dot{\Psi}_1,f_1,g_1), (\dot{A}_2,\dot{\Phi}_2,\dot{B}_2,\dot{\Psi}_2,f_2,g_2)\bigg)\\
:=&-\int_X\mathrm{Tr}(A_1\wedge\ast A_2+\Phi_1\wedge\ast \Phi_2+B_1\wedge\ast B_2+\Psi_1\wedge \ast\Psi_2)\\
&+\mathrm{Re}\bigg(\int_X(f_1,f_2)_{h^\vee_1\otimes h_2}+(g_1,g_2)_{h_1\otimes h^\vee_2}\bigg)
\end{align*}
Here $\mathrm{Re}\bigg(\int_X(f_1,f_2)_{h^\vee_1\otimes h_2}+(g_1,g_2)_{h_1\otimes h^\vee_2}\bigg)$ is the real part of $\int_X(f_1,f_2)_{h^\vee_1\otimes h_2}+(g_1,g_2)_{h_1\otimes h^\vee_2}$.\par

We next define a $\mathbb{R}$-linear morphisms $I:T_xM\to T_x M.$ 
We note that we now regard $A(\mathrm{Hom}(E_1,E_2))$ and $A(\mathrm{Hom}(E_2,E_1))$ as $\mathbb{R}$-vector spaces.
\begin{align*}
I(a)&=I((\dot{A}_1,\dot{\Phi}_1,\dot{B}_1,\dot{\Psi}_1,f_1,g_1)):=(\ast \dot{A}_1,-\ast\dot{\Phi}_1,\ast \dot{B}_1,-\ast\dot{\Psi}_1, \sqrt{-1}f_1,\sqrt{-1}g_1).
\end{align*}
It is straightforward to check that $I^2=-1$.
Therefore, $(T_xM, I)$ is an almost complex vector space.
 $I$ defines a sympletic structure $\omega_I$ on $T_xM$ such that 
 \begin{align*}
\omega_I(a,b):=&g(Ia,b)\\
=&\int_X\mathrm{Tr}(-A_1\wedge\ast A_2+\Phi_1\wedge\ast \Phi_2-B_1\wedge\ast B_2+\Psi_1\wedge \ast\Psi_2)\\
&+\mathrm{Re}\bigg(\sqrt{-1}\int_X(f_1,f_2)_{h^\vee_1\otimes h_2}+(g_1,g_2)_{h_1\otimes h^\vee_2}\bigg),\\
\end{align*}
Since the symplectic forms do not depend on the base point $x$, $\omega_I$ defines a symplectic form on $M$.\par
We now define the moments map $\mu_I$ with respect to $\omega_I$ and the group action $\mathcal{G}_1\times \mathcal{G}_2$.
We regard
\begin{equation*}
A(\mathfrak{u}(E_1))\oplus A(\mathfrak{u}(E_2))
\end{equation*}
as the Lie algebra of $\mathcal{G}_1\times \mathcal{G}_2$ and regard 
\begin{equation*}
A^2(\mathfrak{u}(E_1))\oplus A^2(\mathfrak{u}(E_2))
\end{equation*}
as the dual of the Lie algebra $A(\mathfrak{u}(E_1))\oplus A(\mathfrak{u}(E_2))$ by the pairing
\begin{equation*}
\begin{array}{rccc}
(,)\colon & A(\mathfrak{u}(E_1))\oplus A(\mathfrak{u}(E_2))\times A^2(\mathfrak{u}(E_1))\oplus A^2(\mathfrak{u}(E_2)) &\longrightarrow& \mathbb{R}                   \\
        & \rotatebox{90}{$\in$}&               & \rotatebox{90}{$\in$} \\
        & (u\oplus v, A\oplus B)                    & \longmapsto   & \int_X \mathrm{Tr}(uA+vB)
\end{array}
\end{equation*}

Let 
\begin{equation*}
x=(\nabla_{h_1},\Phi_1,\nabla_{h_2},\Phi_2,\phi,\psi) \in M.
\end{equation*}
We define 
\begin{align*}
\mu_I:M\to A^2(\mathfrak{u}(E_1))\oplus A^2(\mathfrak{u}(E_2))
\end{align*}
as
\begin{align*}
\mu_I(x):=(
F_{h_1}-\Phi_1\wedge\Phi_1-\sqrt{-1}(\phi^\ast\circ \phi)\omega Id_{E_1}&-\sqrt{-1}(\psi\circ \psi^\ast)\omega Id_{E_1},\\
&F_{h_2}-\Phi_2\wedge\Phi_2+\sqrt{-1}(\phi\circ \phi^\ast)\omega Id_{E_2}+\sqrt{-1}(\psi^\ast\circ \psi)\omega Id_{E_2}).
\end{align*}

We can prove that $\mu_I$ is actually a moment map by using the same argument as in \cite[Lemma 2.2]{G2}.
Then it is easy to check from the definition that the space of the solution of the doubly-coupled $\tau$-vortex equation is 
\begin{equation*}
N\cap \mu^{-1}_I(-2\pi\sqrt{-1}\tau Id_{E_1}, -2\pi\sqrt{-1}\tau' Id_{E_2}).
\end{equation*}

We define the moduli space of the doubly-coupled $\tau$-vortex equation $\mathcal{M}_{\tau}$ as
\begin{equation*}
\mathcal{M}_\tau:=N\cap \mu^{-1}_I(-2\pi\sqrt{-1}\tau Id_{E_1}, -2\pi\sqrt{-1}\tau' Id_{E_2})/\mathcal{G}_1\times \mathcal{G}_2.
\end{equation*}

We hope to study the geometry of $\mathcal{M}_\tau$ in the future.\par
It is natural to expect $\mathcal{M}_\tau$ to be a K\"ahler manifold from the moment map picture.
We can moreover expect $\mathcal{M}_\tau$ to be a hyperK\"ahler manifold from the following observation.
Let $x\in M$ and $a\in T_xM$ as above.
We define $\mathbb{R}$-linear maps $J,K:T_xM\to T_xM $ as 
\begin{align*}
J(a)&=J((\dot{A}_1,\dot{\Phi}_1,\dot{B}_1,\dot{\Psi}_1,f_1,g_1)):=(-\dot{\Phi}_1,\dot{A}_1,-\dot{\Psi}_1,\dot{B}_1,-g^\ast_1,f^\ast_1),\\
K(a)&=K((\dot{A}_1,\dot{\Phi}_1,\dot{B}_1,\dot{\Psi}_1,f_1,g_1)):=(-\ast \dot{\Phi}_1,-\ast\dot{A}_1,-\ast \dot{\Psi}_1,-\ast\dot{B}_1, -\sqrt{-1}g^\ast_1,\sqrt{-1}f^\ast_1).
\end{align*}
Then it is easy to check that $J^2=K^2=-Id$ and $K=IJ$.
Hence $I,J,K$ satisfy the quaternion relation.
The author expects that $I,J,K$ defines a hyperK\"ahler structure on $\mathcal{M}_\tau$.

\section{Dimensional Reduction}\label{sec 4}
Let $(X,\omega)$ be a compact Riemann surface and $\mathbb{P}^1$ be a complex projective line. 
We fix these spaces throughout this section. \par
In this section, we show that we obtain a solution of the doubly-coupled $\tau$-vortex equation on $X$ by dimensional reduction of an $SU(2)$-invariant Higgs bundle over $X\times \mathbb{P}^1$ equipped with an invariant Hermitian-Einstein metric.
This gives us a route to show that the existence of the solution of the doubly-coupled $\tau$-vortex equation for certain $\tau$ is equivalent to the stability of Higgs quadruplets.\par
We first recall invariant connections, invariant Higgs bundles.
After, we prove that we obtain a solution of the doubly-coupled $\tau$-vortex equation by dimensional reduction.\par
We freely use the notation of Section \ref{sec2}. 

\subsection{Invariant Connection}
Let $M$ be a smooth manifold and $G$ be a compact Lie group acting on $M$. Let $V$ be a smooth complex vector bundle over $M$. 
We say that $V$ is a $G$-equivalent vector bundle if $G$ acts on $V$ and the action covers the action on $M$.
Let $h_V$ be a hermitian metric on $V$.
We say that $h_V$ is $G$-invariant if the action of $G$ is isometric with respect to $h_V$.
We say a pair $(V, h_V)$ is a $G$-invariant hermitian vector bundle if $V$ is $G$-equivalent and $h_V$ is $G$-invariant.\par
We now focus on a more specific situation. 
In particular, we focus on the case of $M=X\times \mathbb{P}^1$ and $G=SU(2)$.
We also assume that $SU(2)$ acts on $X$ trivially and acts on $\mathbb{P}^1$ as in the Section \ref{sec2}.
Let $p:X\times \mathbb{P}^1\to X$ and $q:X\times \mathbb{P}^1\to  \mathbb{P}^1$ be the projections.\par
We denote the smooth complex line bundle of degree $n$ on $\mathbb{P}^1$ as $H^n$.
We set $H:=H^1$.
Note that $H^n\simeq \mathcal{O}_{\mathbb{P}^1}(n)$ holds as smooth bundles.
The metric $h^{(n)}$ (See Section \ref{sec2}) induces a metric on $H^n$. 
We also denote this metric $h^{(n)}$.
Since $(\mathcal{O}_{\mathbb{P}^1}(n),h^{(n)})$ is an $SU(2)$-invariant hermitian vector bundle, so as $(H^n,h^{(n)})$.\par
We have a clear description of the structure of $SU(2)$-invariant hermitian vector bundles, thanks to \cite{G2}.

\begin{proposition}[{\cite[Proposition 3.1, 3.2]{G2}}]\label{prop 4.1}
Let $F$ be an $SU(2)$-equivalent vector bundle on $X\times \mathbb{P}^1.$ Then $F$ can be uniquely decomposed up to isomorphisms as 
\begin{equation*}
F=\bigoplus_i F_i
\end{equation*}
with $F_i=p^\ast E_i\otimes q^\ast H^{n_i}$, $E_i$ is a vector bundle over $X$ and $n_i\in\mathbb{Z}$ are all different.\par
Let $(F,h)$ be an $SU(2)$-invariant hermitian vector bundle. Let 
\begin{equation*}
F=\bigoplus_i F_i=\bigoplus_i p^\ast E_i\otimes q^\ast H^{n_i}
\end{equation*}
be the decomposition given as above. Then \par
(1) The vector bundles $F_i$ are $SU(2)$-invariantly othogonal to each other; $h=\bigoplus_i\hat{h}_i$ with $\hat{h}_i$ with an $SU(2)$-invariant metric on $F_i$.\par
(2) Each $\hat{h}_i$ has the form $\hat{h}_i=p^\ast h_i\otimes q^\ast h^{(n_i)}$, where $h_i$ is the hermitian metric on $E_i$.
\end{proposition}
Let $(F,h)$ be an $SU(2)$-invariant hermitian vector bundle over $X\times\mathbb{P}^1$.
Assume that it is decomposed as 
\begin{equation*}
F=\bigoplus^m_{i=1} F_i\,\, \mathrm{and}\,\, h=\bigoplus^m_{i=1} h_i.
\end{equation*}
Let $D$ be a connection on $F$.
We do not assume now $D$ to be $SU(2)$-equivalent.
Recall that any connection $D$ has the decomposition 
\begin{equation*}
D=\nabla_h+\sqrt{-1}\Phi
\end{equation*}
such that $\nabla_h\in \mathcal{A}_h$ and $\Phi\in A^{1}(\mathfrak{u}(F))$. 
With respect to the decomposition $F=\bigoplus^m_{i=1} F_i$, $\nabla_h$ and $\Phi$ has the decomposition as 
\begin{align*}
\nabla_h=
\begin{pmatrix}
\nabla_{h_1} & A_{12} & \dots  & A_{1m} \\
  A_{21} & \nabla_{h_2} & \dots  & A_{2m} \\
  \vdots & \vdots & \ddots & \vdots \\
  A_{m1} & A_{m2} & \dots  & \nabla_{h_m}
\end{pmatrix}
\end{align*}
and 
\begin{align*}
\Phi=
\begin{pmatrix}
\Phi_{11} & \Phi_{12} & \dots  & \Phi_{1m} \\
  \Phi_{21} & \Phi_{22} & \dots  & \Phi_{2m} \\
  \vdots & \vdots & \ddots & \vdots \\
  \Phi_{m1} & \Phi_{m2} & \dots  & \Phi_{mm}
\end{pmatrix}
\end{align*}
where $A_{ij},\Phi_{ij}\in A^1(\mathrm{Hom}(F_j,F_i))$.
The following holds.
\begin{proposition}\label{prop 4.2}
(1) $\nabla_{h_i}$ is a metric connection for $(F_i,h_i)$.\par
(2) $A_{ij}$ and $-A_{ji}$ are self-adjoint with respect to the metric $h$:
 \begin{equation*}
 h(A_{ji}s,t)+h(s,A_{ij}t)=0\,\,\mathrm{for\, all}\,\, s\in A(F_i), t\in A(F_j).
 \end{equation*}
 \par
 (3) $\Phi_{ij}$ and $\Phi_{ij}$ are self-adjoint with respect to the metric $h$.
  \end{proposition} 
 \begin{proof}
 (1) and (2) are proved in \cite[Proposition 3.4]{G2}.\par
 (3) Let $s\in A(F_i)$ and $t\in A(F_j)$. Since $\Phi$ is skew-symmetric, we have 
 \begin{equation*}
 h(\Phi s,t)+h(s,\Phi t)=0.
 \end{equation*} 
Therefore
\begin{align*}
0&=h(\Phi s,t)+h(s,\Phi t)\\
&=h(\sum_{k}\Phi_{ki}s,t)+h(s,\sum_{l}\Phi_{lj}t)\\
&=h_i(\Phi_{ji}s,t)+h_j(s,\Phi_{ij}t).
\end{align*}
The third equation follows from the decomposition of the $h$.
 \end{proof}
 The following case is relevant for us.
 Let $(E_1,h_1)$ and $(E_2,h_2)$ be hermitian vector bundles on $X$. 
 Let $F$ be the vector bundle over $X\times \mathbb{P}^1$ defined as 
 \begin{equation*}
 F=F_1\oplus F_2=p^\ast E_1\oplus p^\ast E_2\otimes q^\ast H^2
 \end{equation*} 
 with a hermitian metric $h=p^\ast h_1\oplus p^\ast h_2\otimes q^\ast h_2$.
Note the total space of $p^\ast E_i$ is $E_i\times \mathbb{P}^1$ and the $SU(2)$-action we are considering is trivial on $E_i$ and the standard one on $\mathbb{P}^1$.
Therefore $(F, h)$ is an $SU(2)$-invariant hermitian vector bundle. \par
Let $D$ be a connection on $F$ and $D=\nabla_h+\sqrt{-1}\Phi$ be the decomposition to a metric connection and a skew-symmetric 1-form.
Recall that by Proposition \ref{prop 4.2}, $D$ has the form 
\begin{align*}
D&=\nabla_h+\sqrt{-1}\Phi\\
&=
\begin{pmatrix}
\nabla_{h_1}& A\\
-A^\ast&\nabla_{h_2}
\end{pmatrix}
+
\sqrt{-1}
\begin{pmatrix}
\Phi_{1}&-B^\ast\\
B&\Phi_{2}
\end{pmatrix}
\end{align*}
where $\nabla_{h_i}\in\mathcal{A}_{h_i}$, $\Phi_i\in A^{1}(\mathfrak{u}(F_i))$, $A\in A^1(\mathrm{Hom}(F_2,F_1)), B\in A^1(\mathrm{Hom}(F_1,F_2))$, and $A^\ast, B^\ast$ are the formal adjoints of $A, B$.\par
We next study the structure of an $SU(2)$-invariant connection on $F$.
Before we proceed, we make clear what an $SU(2)$-invariant connection means.\par
We go back to the general settings. 
Let $M$ be a smooth manifold with a group $G$ acting on it. 
Let $\pi:V\to M$ be a smooth $G$-equivalent complex vector bundle over $M$.
This is equivalent to assume that for every $g\in G$, we have a map $\Psi_g:V\to V$ such that $\Psi_g:V_x=\pi^{-1}(x)\to V_{g\cdot x}$ holds and $\Psi_g|_{V_x}$ is a linear isomorphism.
For every $g\in G$, we can define a pullback bundle $g^\ast V$. $\Psi_g$ induces a vector bundle isomorphism $\hat{\Psi}_g;V\to g^\ast V$.
Let $D$ be a connection on $V$.
Then $g$ defines a pullback connection $g^\ast D$ on $g^\ast V$.
We say that $D$ is a $G$-invariant connection if $(\hat{\Psi}_g)^{-1}\circ g^\ast D\circ \hat{\Psi}_g=D$ holds.\par
We go back to the above setting. 
Let $T^\ast (X\times \mathbb{P}^1), T^\ast X, T^\ast\mathbb{P}^1$ be the smooth cotangent bundles of $X\times\mathbb{P}^1, X, \mathbb{P}^1$. Note that 
\begin{equation*}
T^\ast\mathbb{P}^1=T^{\ast(1,0)}\mathbb{P}^1\oplus T^{\ast(0,1)}\mathbb{P}^1\simeq\mathcal{O}_{\mathbb{P}^1}(-2)\oplus \mathcal{O}_{\mathbb{P}^1}(2)\simeq H^{-2}\oplus H^2
\end{equation*}
 holds as smooth bundles.
We also note that 
\begin{equation*}
T^\ast(X\times \mathbb{P}^1)=p^\ast T^\ast X\oplus q^\ast T^\ast \mathbb{P}^1
\end{equation*} 
holds.
Since $T^\ast \mathbb{P}^1\simeq \mathcal{O}_{\mathbb{P}^1}(-2)\oplus \mathcal{O}_{\mathbb{P}^1}(2)$, $SU(2)$ acts on $T^\ast \mathbb{P}^1$. 
This induces an $SU(2)$-action on $T^\ast(X\times \mathbb{P}^1)$.
 Recall that we assumed $SU(2)$ acts on $X$ trivially. 
 Therefore $T^\ast(X\times \mathbb{P}^1)$ is an $SU(2)$-equivalent bundle.\par
 Let $E$ be an $SU(2)$-equivalent vector bundle over $X\times\mathbb{P}^1$. 
 We say a section $a\in A^1(E)$ is $SU(2)$-invariant if it is an $SU(2)$-invariant section of the vector bundle $T^\ast(X\times \mathbb{P}^1)\otimes E.$\par
 Let $D$ be an $SU(2)$-invariant connection on $F$ and let $\nabla_{h_i},\Phi_i,A,B$ be as above.
 It is straightforward to check that the $SU(2)$-invariance of $D$ implies the $SU(2)$-invariance of $\nabla_{h_i},\Phi_i,A,B.$\par
The following result was essentially proved in \cite[Proposition 3.5]{G2}. 
Some part of it is slightly modified from the original one, and some results are added for our purpose.

 \begin{proposition}[{\cite[Proposition 3.5]{G2}}]\label{prop 4.3}
(1) Let $\nabla_{p^\ast h_1}$ be an $SU(2)$-invariant unitary connection on $(F_1,p^\ast h_1)$. 
Then there exists an unitary connection $\nabla_{h_1}$ on $(E_1,h_1)$ such that $\nabla_{p^\ast h_1}=p^\ast\nabla_{h_1}$.\par
(2) Let $\nabla_{p^\ast h_2\otimes q^\ast h^{(2)}}$ be an $SU(2)$-invariant unitary connection on $(F_2, p^\ast h_2\otimes q^\ast h^{(2)})$. 
Then there exists an unitary connection $\nabla_{h_2}$ on $(E_2,h_2)$ such that $\nabla_{p^\ast h_1}=p^\ast\nabla_{h_1}\otimes Id+Id\otimes q^\ast \nabla_{h^{(2)}}$.
Here $\nabla_{h^{(2)}}$ is the unique unitary connection of $H^{2}$. \par
(3) $A=p^\ast\psi\otimes q^\ast \alpha$ where $\phi\in A(E_1\otimes E^\ast_2)$ and $\alpha$ is the the unique $SU(2)$-invariant section of $A^1(H^{-2})$ up to constants.
 In particular $\alpha\in A^{0,1}(H^{-2})$.\par
(4) $B=p^\ast \phi\otimes q^\ast\beta$ where $\psi\in A(E^\ast_1\otimes E_2)$ and $\beta$ is the unique $SU(2)$-invariant section of $A^1(H^{2})$ up to constants.
 In particular $\beta\in A^{1,0}(H^{2})$.\par
(5) For each $\Phi_i\in A^1(\mathfrak{u}(F_i))$, there exists $\Psi\in A^1(\mathfrak{u}(E_i))$ such that $\Phi_i=p^\ast\Psi_i$.\par
In particular every $SU(2)$-invariant connection $D$ on $F$ has the form
\begin{align*}
D=
\begin{pmatrix}
p^\ast\nabla_{h_1}& p^\ast \psi\otimes q^\ast \alpha \\
-(p^\ast \psi\otimes q^\ast \alpha)^\ast& p^\ast\nabla_{h_2}\otimes Id +Id\otimes q^\ast\nabla_{h^{(2)}}
\end{pmatrix}
+
\sqrt{-1}
\begin{pmatrix}
p^\ast\Psi_1& -(p^\ast \phi\otimes q^\ast\beta)^\ast\\
p^\ast \phi\otimes q^\ast\beta& p^\ast\Psi_2
\end{pmatrix}
\end{align*}
where $\nabla_{h_i}$ are unitary connections on $(E_i,h_i)$, $\Psi_i\in A^1(\mathfrak{u}(E_i))$, and $\phi\in A(\mathrm{Hom}(E_1,E_2)),\psi\in A(\mathrm{Hom}(E_2,E_1))$.
\end{proposition}
\begin{proof}
In \cite{G2}, (1), (2) were proved for general $SU(2)$-invariant connections, and (3) was proved. \par
We give the outline of the proof of $(1), (2)$. Let $\nabla_{h_1}$ be an unitary connection of $(E_1,h_1).$
Then it is clear that $p^\ast\nabla_{h_1}$ is a metric connection of $(F_1=p^\ast E_1,p^\ast h_1)$ and $p^\ast\nabla_{h_1}$ is $SU(2)$-equivalent.
Let $\nabla_{p^\ast h_1}$ be a unitary connection of $(F_1,p^\ast h_1)$.
Note that $\nabla_{p^\ast h_1}-p^\ast\nabla_{h_1}=C\in A^1(\mathfrak{u}(F_1))$.
We show that if $\nabla_{p^\ast h_1}$ is $SU(2)$-equivalent, there exists $C'\in A^1(\mathfrak{u}(E))$ such that $C= p^\ast C'.$
This induces the proof of (1).
\par
Suppose $\nabla_{p^\ast h_1}$ is $SU(2)$-equivalent.
Then $C$ is also an $SU(2)$-equivalent section.
Let $A^1(\mathrm{End}F_1)^{SU(2)}$ be the space of $SU(2)$-equivalent section of $\mathrm{End}F_1$.
It was shown in \cite{G2} that $A^1(\mathrm{End}F_1)^{SU(2)}=p^\ast(A^1(\mathrm{End}E))$.
 Therefore, we have a $C'\in A^1(\mathrm{End}E_1)$ such that $C=p^\ast C'$.
 Since $C\in A^1(\mathfrak{u}(F_1))$, $C'\in A^1(\mathfrak{u}(E_1))$.
\par
(2) can be proved by tensoring $q^\ast H^{-2}$ to $F_2$ and the connection $q^\ast \nabla_{h^{(-2)}}$ which is the dual of $q^\ast \nabla_{h^{(2)}}$ to $\nabla_{p^\ast h_2\otimes q^\ast h^{(2)}}$ and apply (1).\par
For the proof of (3), see \cite{G2}.
We only give some comments why $\alpha\in A^{0,1}(H^{-2})$.
First
\begin{align*}
A^{1}(H^{-2})&=A(T^\ast\mathbb{P}^1\otimes H^{-2})\\
&=A((H^2\oplus H^{-2}\otimes )H^{-2})\\
&=A(H^2\otimes H^{-2})\otimes A(H^{-2}\otimes H^{-2}).
\end{align*}
Then by the same argument as in Proposition \ref{prop 2.1}, $A(H^{-2}\otimes H^{-2})$ cannot have an $SU(2)$-invariant section.
Therefore
\begin{equation*}
\alpha\in A(H^2\otimes H^{-2})=A^{0,1}(H^{-2}).
\end{equation*}\
The equation follows from $T^{\ast(0,1)}\mathbb{P}^1\simeq H^2$. The uniqueness of $\alpha$ follows from Proposition \ref{prop 2.1}.
\par
(4) We have the following isomorphisms
\begin{align*}
A^1(\mathrm{Hom}(F_1,F_2))=&A^1(F^\vee_1\otimes F_2)\\
=&A^1(p^\ast(E^\vee_1\otimes E_2)\otimes q^\ast H^2)\\
=& A((p^\ast T^\ast X\oplus q^\ast T^\ast \mathbb{P}^1)\otimes p^\ast(E^\ast_1\otimes E_2)\otimes q^\ast H^2)\\
\simeq &A(p^\ast(E^\vee_1\otimes E_2\otimes T^\ast X)\otimes q^\ast H^2)\\
&\oplus A(p^\ast(E^\vee_1\otimes E_2)\otimes q^\ast H^{-2}\otimes q^\ast H^2)\\
&\oplus A(p^\ast(E^\vee_1\otimes E_2)\otimes q^\ast H^2\otimes q^\ast H^2)
\end{align*}
Suppose $s\in A(p^\ast(E^\ast_1\otimes E_2\otimes T^\ast X)\otimes q^\ast H^2)\,\,(\mathrm{resp.}\,\,A(p^\ast(E^\ast_1\otimes E_2)\otimes q^\ast H^2\otimes q^\ast H^2) )$ is an $SU(2)$-invariant section.
Then it indues an $SU(2)$-invariant section of $A(H^2)\,\,(\mathrm{resp.}\,\,A(H^2\otimes H^2))$ by restriction.
However, such sections do not exist by Proposition  \ref{prop 2.1}.\par
Therefore $B\in A(p^\ast(E^\ast_1\otimes E_2)\otimes q^\ast H^{-2}\otimes q^\ast H^2)$.
The $SU(2)$-invariance of $B$ implies that it has a form
\begin{equation*}
B=p^\ast \phi\otimes  q^\ast \beta
\end{equation*}
where $\phi\in A(E^\ast_1\otimes E_2)$ and $\beta\in A^1(H^{2})$.\par
For the same reason with $\alpha\in A^{0,1}(H^{-2})$, $\beta\in A^{1,0}(H^{2}).$
The uniqueness of $\beta$ follows from Proposition \ref{prop 2.1}.\par
(5) We already know that there exists a $\Psi_1\in A^1(\mathfrak{u}(E_1))$ such that  $\Phi_1=p^\ast \Psi_1$.
We give a proof for $\Phi_2$.
First, we note that 
\begin{align*}
A(\mathrm{End}F_2)&=A(F_2^\vee\otimes F_2)\\
&=A((p^\ast E_2\otimes q^\ast H^{2})^\vee\otimes p^\ast E_2\otimes q^\ast H^{2}))\\
&=A(p^\ast (E_2^\vee\otimes E_2)).
\end{align*}
Therefore, we can apply the above argument to show that there exists a $\Psi_2\in A^1(\mathfrak{u}(E_2))$ such that  $\Phi_2=p^\ast \Psi_2$.
\end{proof}

Recall that in Section \ref{sec2}, we studied the set 
\begin{equation*}
M=\mathcal{A}_{h_1}\times A^1(\mathfrak{u}(E_1))\times\mathcal{A}_{h_2}\times A^1(\mathfrak{u}(E_2))\times  A(\mathrm{Hom}(E_1,E_2))\times A(\mathrm{Hom}(E_2,E_1)).
\end{equation*}
Let $(\nabla_{h_1},\Psi_1,\nabla_{h_2},\Psi_2, \phi,\psi)\in M$.
We define the connection $D$ on $F$ as 
\begin{align*}
D=
\begin{pmatrix}
p^\ast\nabla_{h_1}& p^\ast \psi\otimes q^\ast \alpha \\
-(p^\ast \psi\otimes q^\ast \alpha)^\ast& p^\ast\nabla_{h_2}\otimes Id +Id\otimes q^\ast\nabla_{h^{(2)}}
\end{pmatrix}
+
\sqrt{-1}
\begin{pmatrix}
p^\ast\Psi_1& -(p^\ast \phi\otimes q^\ast\beta)^\ast\\
p^\ast \phi\otimes q^\ast\beta& p^\ast\Psi_2
\end{pmatrix}.
\end{align*}
This connection is clearly an $SU(2)$-invariant connection.
Let 
\begin{align*}
\mathcal{A}^{SU(2)}_{\mathrm{Conn},F}:=\{ D\,\,|\,\, D\,\, \textrm{is an} \,\, SU(2)\textrm{-invariant connection of}\,\, F\}.
\end{align*}
Combining Proposition \ref{prop 4.3} and the above argument, we have the following.

\begin{proposition}\label{prop 4.4}
Let $D\in\mathcal{A}^{SU(2)}_{\mathrm{Conn},F}$ and let 
\begin{align*}
D=
\begin{pmatrix}
p^\ast\nabla_{h_1}& p^\ast \psi\otimes q^\ast \alpha \\
-(p^\ast \psi\otimes q^\ast \alpha)^\ast& p^\ast\nabla_{h_2}\otimes Id +Id\otimes q^\ast\nabla_{h^{(2)}}
\end{pmatrix}
+
\sqrt{-1}
\begin{pmatrix}
p^\ast\Psi_1& -(p^\ast \phi\otimes q^\ast\beta)^\ast\\
p^\ast \phi\otimes q^\ast\beta& p^\ast\Psi_2
\end{pmatrix}
\end{align*}
be the decomposition we obtained in Proposition \ref{prop 4.3}.\par
We define a map
\begin{align*}
\iota_h:\mathcal{A}^{SU(2)}_{\mathrm{Conn,F}}\to M
\end{align*}
as 
\begin{equation*}
\iota_h(D):=(\nabla_{h_1},\Psi_1,\nabla_{h_2},\Psi_2, \phi,\psi)\in M.
\end{equation*}
Then $\iota$ is bijective.
\end{proposition}
\subsection{Invariant Higgs Bundles}
We use the notations of the previous section.\par
Let $M$ be a complex manifold with a group $G$ acting on it.
A $G$-equivalent Higgs bundle $(V,\overline\partial_V,\theta_V)$ over $M$ is a Higgs bundle over $M$ such that $V$ is a $G$-equivalent complex vector bundle over $M$, $\overline\partial_V$ is a $G$-invariant $(0,1)$-type differential operator, and $\theta_V$ is a $G$-invariant  section of $A^{1,0}(\mathrm{End}V)$.\par

In this section, we study $SU(2)$-invariant Higgs bundles on
 $F=F_1\oplus F_2=p^\ast E_1\oplus p^\ast E_2\otimes q^\ast H^2=p^\ast E_1\oplus p^\ast E_2\otimes q^\ast \mathcal{O}_{\mathbb{P}^1}(2)$, using an $SU(2)$-invariant hermitian metric $h=p^\ast h_1\oplus p^\ast h_2\otimes q^\ast h^{(2)}$ and $\iota_h$.
We use $\mathcal{O}_{\mathbb{P}^1}(2)$ instead of $H^2$ since we deal with holomorphic objects.\par
Let 
\begin{align*}
A^{SU(2)}_{\mathrm{Higgs},F}:=
\{
(\overline\partial_F,\theta_F)\,\,|\,\, (F,\overline\partial_F,\theta_F)\,\,\textrm{is an} \,\,SU(2)\textrm{-invariant Higgs bundle}\}
\end{align*}
and let $\mathcal{A}^{SU(2)}_{\mathrm{Higgs},h,F}\subset \mathcal{A}^{SU(2)}_{\mathrm{Conn},F}$ be the subset 
\begin{equation*}
\mathcal{A}^{SU(2)}_{\mathrm{Higgs},h,F}:=\{ D=\nabla_h+\sqrt{-1}\Phi\in \mathcal{A}^{SU(2)}_{\mathrm{Conn}}\,\,|\,\, (\nabla^{0,1}_h+\sqrt{-1}\Phi^{1,0})^2=0\}.
\end{equation*}

\par
Let $(\overline\partial_F,\theta_F)\in A^{SU(2)}_{\mathrm{Higgs},F}$, $\partial_h$ be the $(1,0)$-part of the Chern connection with respect to $\overline\partial_F$, and $\theta^\dagger_{h}$ be the formal adjoint of $\theta_F$ with respect to $h$.
Since $(F,\overline\partial_F,\theta_F)$ is an $SU(2)$-invariant Higgs bundle and $h$ is an $SU(2)$-invariant metric, we have an $SU(2)$-invariant connection
\begin{align*}
D:=&\partial_h+\overline\partial_F+\theta_F+\theta^\dagger_{h}\\
=&\partial_h+\overline\partial_F+\sqrt{-1}(-\sqrt{-1}\theta_F-\sqrt{-1}\theta^\dagger_{h})
\end{align*}
We note that $\partial_h+\overline\partial_F$ is a metric connection and $-\sqrt{-1}\theta_F-\sqrt{-1}\theta^\dagger_{h}$ is a skew-symmetric 1-form with respect to $h$.
Hence, we have a map
\begin{equation*}
\begin{array}{rccc}
\kappa\colon & A^{SU(2)}_{\mathrm{Higgs},F} &\longrightarrow& \mathcal{A}^{SU(2)}_{\mathrm{Higgs},h,F}                   \\
        & \rotatebox{90}{$\in$}&               & \rotatebox{90}{$\in$} \\
        &   (\overline\partial_F,\theta_F)                  & \longmapsto   & D=\partial_h+\overline\partial_F+\sqrt{-1}(-\sqrt{-1}\theta_F-\sqrt{-1}\theta^\dagger_{h})
        \end{array}.
\end{equation*}
$\kappa$ is a bijective map since metric connections $\nabla_h$ and skew-symmetric 1-forms $\Phi$ are determined by its $(0,1)$- and $(1,0)$-part.
Therefore, to study $A^{SU(2)}_{\mathrm{Higgs},F}$ it is enough to study $\mathcal{A}^{SU(2)}_{\mathrm{Higgs},h,F}$
\par

Let $D\in \mathcal{A}^{SU(2)}_{\mathrm{Higgs},h,F}$.
Then by Proposition \ref{prop 4.3}, $D$ has the decomposition

\begin{align*}
D&=\nabla_h+\sqrt{-1}\Phi\\
&=\begin{pmatrix}
p^\ast\nabla_{h_1}& p^\ast \psi\otimes q^\ast \alpha \\
-(p^\ast \psi\otimes q^\ast \alpha)^\ast& p^\ast\nabla_{h_2}\otimes Id +Id\otimes q^\ast\nabla_{h^{(2)}}
\end{pmatrix}
+
\sqrt{-1}
\begin{pmatrix}
p^\ast\Psi_1& -(p^\ast \phi\otimes q^\ast\beta)^\ast\\
p^\ast \phi\otimes q^\ast\beta& p^\ast\Psi_2
\end{pmatrix}
\end{align*}

where $\nabla_h$ is an unitary connection on $F$, $\Phi\in A^1(\mathfrak{u}(F))$, $\nabla_{h_i}$ are unitary connections on $(E_i,h_i)$, $\Psi_i\in A^1(\mathfrak{u}(E_i))$, and $\phi\in A(\mathrm{Hom}(E_1,E_2)),\psi\in A(\mathrm{Hom}(E_2,E_1))$.\par
Since $\alpha\in A^{0,1}(\mathcal{O}_{\mathbb{P}^1}(2))$ and $\beta\in A^{1,0}(\mathcal{O}_{\mathbb{P}^1}(-2))$, $(p^\ast \psi\otimes q^\ast \alpha)^\ast\in A^{1,0}(\mathrm{Hom}(F_2,F_1))$ and $(p^\ast \phi\otimes q^\ast\beta)^\ast\in A^{1,0}(\mathrm{Hom}(F_1,F_2))$.
 \par
 
Therefore the $(0,1)$-part $\nabla_h$ and the $(1,0)$-part of the $\sqrt{-1}\Phi$ can be expressed as 
\begin{align*}
\nabla^{0,1}_h
&=
\begin{pmatrix}
p^\ast\nabla^{0,1}_{h_1}& p^\ast \psi\otimes q^\ast \alpha \\
0& p^\ast\nabla^{0,1}_{h_2}\otimes Id +Id\otimes q^\ast\nabla^{0,1}_{h^{(2)}}
\end{pmatrix}\\
\sqrt{-1}\Phi^{1,0}&=\sqrt{-1}
\begin{pmatrix}
p^\ast\Psi^{1,0}_1& 0\\
p^\ast \phi\otimes q^\ast\beta& p^\ast\Psi^{1,0}_2
\end{pmatrix}
\end{align*}

\begin{proposition}\label{prop 4.5}
$(\nabla^{0,1}_h,\sqrt{-1}\Phi^{1,0})$ defines a Higgs bundle structure on $F$ (i.e. $(\nabla^{0,1}_h+\sqrt{-1}\Phi^{1,0})^2=0$) if and only if the following equations are satisifed 
\begin{itemize}
\item[1.] $(\nabla^{0,1}_{h_1}+\sqrt{-1}\Psi^{1,0}_1)^2=(\nabla^{0,1}_{h_2}+\sqrt{-1}\Psi^{1,0}_2)^2=0$.
\item[2.] $\phi\circ(\nabla^{0,1}_{h_1}+\sqrt{-1}\Psi^{1,0}_1)=(\nabla^{0,1}_{h_2}+\sqrt{-1}\Psi^{1,0}_2)\circ\phi$.
\item[3.] $(\nabla^{0,1}_{h_1}+\sqrt{-1}\Psi^{1,0}_1)\circ\psi=\psi\circ(\nabla^{0,1}_{h_2}+\sqrt{-1}\Psi^{1,0}_2)$.
\item[4.] $\phi\circ\psi=\psi\circ\phi=0$.
\end{itemize}
\end{proposition}
\begin{proof}
$(\nabla^{0,1}_h,\sqrt{-1}\Phi^{1,0})$ defines a Higgs bundle structure on $F$ if and only if 
$(\nabla^{0,1}_h)^2=\nabla^{0,1}_h\Phi^{1,0}+\Phi^{1,0}\nabla^{0,1}_h=\Phi^{1,0}\wedge \Phi^{1,0}=0$ holds.
We check that this condition is equivalent to 1., 2., 3., 4. by direct calculation.\par
\begin{align*}
(\nabla^{0,1}_h)^2&=
\begin{pmatrix}
p^\ast\nabla^{0,1}_{h_1}& p^\ast \psi\otimes q^\ast \alpha \\
0& p^\ast\nabla^{0,1}_{h_2}\otimes Id +Id\otimes q^\ast\nabla^{0,1}_{h^{(2)}}
\end{pmatrix}
\begin{pmatrix}
p^\ast\nabla^{0,1}_{h_1}& p^\ast \psi\otimes q^\ast \alpha \\
0& p^\ast\nabla^{0,1}_{h_2}\otimes Id +Id\otimes q^\ast\nabla^{0,1}_{h^{(2)}}
\end{pmatrix}\\
&=
\begin{pmatrix}
(p^\ast\nabla^{0,1}_{h_1})^2& p^\ast\nabla^{0,1}_{h_1}(p^\ast \psi\otimes q^\ast \alpha)+(p^\ast \psi\otimes q^\ast \alpha)p^\ast\nabla^{0,1}_{h_2}\otimes Id +Id\otimes q^\ast\nabla^{0,1}_{h^{(2)}} \\
0& (p^\ast\nabla^{0,1}_{h_2}\otimes Id +Id\otimes q^\ast\nabla^{0,1}_{h^{(2)}})^2
\end{pmatrix}\\
&=
\begin{pmatrix}
(p^\ast\nabla^{0,1}_{h_1})^2& p^\ast(\nabla^{0,1}_{h_1}\psi-\psi\nabla^{0,1}_{h_2})\otimes q^\ast \alpha \\
0& (p^\ast\nabla^{0,1}_{h_2})^2
\end{pmatrix}.
\end{align*}
$(p^\ast\nabla^{0,1}_{h_1})^2=(p^\ast\nabla^{0,1}_{h_2})^2=0$ always hold since $X$ is a Riemann surface and the connection forms of $p^\ast\nabla^{0,1}_{h_i}$ has only diffirential forms of  $X$.
Since $\alpha$ is everywhere a non-zero section of $A^{0,1}(\mathcal{O}_{\mathbb{P}^1}(-2))$, $(\nabla^{0,1}_h)^2=0$ is equivalent to  

\begin{equation}\label{prop 4.5 eq 1}
\nabla^{0,1}_{h_1}\psi-\psi\nabla^{0,1}_{h_2}=0.
\end{equation}
 
\begin{align*}
\Phi^{1,0}\wedge\Phi^{1,0}&=
\begin{pmatrix}
p^\ast\Psi^{1,0}_1& 0\\
p^\ast \phi\otimes q^\ast\beta& p^\ast\Psi^{1,0}_2
\end{pmatrix}
\wedge
\begin{pmatrix}
p^\ast\Psi^{1,0}_1& 0\\
p^\ast \phi\otimes q^\ast\beta& p^\ast\Psi^{1,0}_2
\end{pmatrix}\\
&=
\begin{pmatrix}
p^\ast\Psi^{1,0}_1\wedge p^\ast\Psi^{1,0}_1& 0\\
p^\ast \phi\otimes q^\ast\beta\wedge p^\ast\Psi^{1,0}_1+p^\ast\Psi^{1,0}_2\wedge p^\ast \phi\otimes q^\ast\beta & p^\ast\Psi^{1,0}_2\wedge p^\ast\Psi^{1,0}_2
\end{pmatrix}\\
&=\begin{pmatrix}
p^\ast(\Psi^{1,0}_1\wedge \Psi^{1,0}_1)& 0\\
p^\ast (-\phi\Psi^{1,0}_1+\Psi^{1,0}_2 \phi)\otimes q^\ast\beta & p^\ast(\Psi^{1,0}_2\wedge \Psi^{1,0}_2)
\end{pmatrix}.
\end{align*}
$p^\ast(\Psi^{1,0}_1\wedge \Psi^{1,0}_1)=p^\ast(\Psi^{1,0}_2\wedge \Psi^{1,0}_2)=0$ always hold since $X$ is a Riemann surface.
Since $\beta$ is everywhere a non-zero section of $A^{1,0}(\mathcal{O}_{\mathbb{P}^1}(2))$, $\Phi^{1,0}\wedge\Phi^{1,0}=0$ is equivalent to 

\begin{equation}\label{prop 4.5 eq 2}
-\phi  \Psi^{1,0}_1+\Psi^{1,0}_2 \phi=0.
\end{equation}

\begin{align*}
\nabla_h^{0,1}\Phi^{1,0}&=
\begin{pmatrix}
p^\ast\nabla^{0,1}_{h_1}& p^\ast \psi\otimes q^\ast \alpha \\
0& p^\ast\nabla^{0,1}_{h_2}\otimes Id +Id\otimes q^\ast\nabla^{0,1}_{h^{(2)}}
\end{pmatrix}
\begin{pmatrix}
p^\ast\Psi^{1,0}_1& 0\\
p^\ast \phi\otimes q^\ast\beta& p^\ast\Psi^{1,0}_2
\end{pmatrix}\\
&=\begin{pmatrix}
p^\ast\nabla^{0,1}_{h_1}p^\ast\Psi^{1,0}_1+p^\ast \psi\otimes q^\ast \alpha\wedge p^\ast \phi\otimes q^\ast\beta &p^\ast \psi\otimes q^\ast \alpha\wedge p^\ast\Psi^{1,0}_2 \\
(p^\ast\nabla^{0,1}_{h_2}\otimes Id +Id\otimes q^\ast\nabla^{0,1}_{h^{(2)}})p^\ast \phi\otimes q^\ast\beta& (p^\ast\nabla^{0,1}_{h_2}\otimes Id +Id\otimes q^\ast\nabla^{0,1}_{h^{(2)}}) p^\ast\Psi^{1,0}_2
\end{pmatrix}\\
&=\begin{pmatrix}
p^\ast\nabla^{0,1}_{h_1}p^\ast\Psi^{1,0}_1+p^\ast (\psi\circ\phi)\otimes q^\ast (\alpha\wedge\beta)& p^\ast(-\psi \Psi^{1,0}_2)\otimes q^\ast \alpha\\
(p^\ast\nabla^{0,1}_{h_2}\otimes Id +Id\otimes q^\ast\nabla^{0,1}_{h^{(2)}})p^\ast \phi\otimes q^\ast\beta&(p^\ast\nabla^{0,1}_{h_2}\otimes Id +Id\otimes q^\ast\nabla^{0,1}_{h^{(2)}}) p^\ast\Psi^{1,0}_2
\end{pmatrix}.
\end{align*}

Since $\alpha\in A^{0,1}(\mathcal{O}_{\mathbb{P}^1}(-2)),\beta\in A^{1,0}(\mathcal{O}_{\mathbb{P}^1}(2))$ and $\mathcal{O}_{\mathbb{P}^1}(-2)^\vee=\mathcal{O}_{\mathbb{P}^1}(2)$, $\alpha\wedge\beta$ can be defined and $\alpha\wedge \beta\in A^{1,1}(\mathbb{P}^1)$.
It is clear from the construction of $\alpha$ and $\beta$ that $\alpha\wedge\beta$ is a non-degenerate differential form.
$\beta\wedge \alpha$ is also defined and is non-degenerate.

\begin{align*}
\nabla_h^{0,1}\Phi^{1,0}&=
\begin{pmatrix}
p^\ast\Psi^{1,0}_1& 0\\
p^\ast \phi\otimes q^\ast\beta& p^\ast\Psi^{1,0}_2
\end{pmatrix}
\begin{pmatrix}
p^\ast\nabla^{0,1}_{h_1}& p^\ast \psi\otimes q^\ast \alpha \\
0& p^\ast\nabla^{0,1}_{h_2}\otimes Id +Id\otimes q^\ast\nabla^{0,1}_{h^{(2)}}
\end{pmatrix}\\
&=\begin{pmatrix}
p^\ast\Psi^{1,0}_1p^\ast\nabla^{0,1}_{h_1}& p^\ast\Psi^{1,0}_1\wedge p^\ast \psi\otimes q^\ast \alpha \\
(p^\ast \phi\otimes q^\ast\beta)p^\ast\nabla^{0,1}_{h_1}& p^\ast \phi\otimes q^\ast\beta\wedge p^\ast \psi\otimes q^\ast \alpha + p^\ast\Psi^{1,0}_2(p^\ast\nabla^{0,1}_{h_2}\otimes Id +Id\otimes q^\ast\nabla^{0,1}_{h^{(2)}})
\end{pmatrix}\\
&=\begin{pmatrix}
p^\ast\Psi^{1,0}_1p^\ast\nabla^{0,1}_{h_1}& p^\ast(\Psi^{1,0}_1 \psi)\otimes q^\ast \alpha\\
(p^\ast \phi\otimes q^\ast\beta)p^\ast\nabla^{0,1}_{h_1}& p^\ast (\phi\circ\psi)\otimes q^\ast(\beta\wedge \alpha) + p^\ast\Psi^{1,0}_2(p^\ast\nabla^{0,1}_{h_2}\otimes Id +Id\otimes q^\ast\nabla^{0,1}_{h^{(2)}}).
\end{pmatrix}
\end{align*}

Therefore $\nabla^{0,1}_h\Phi^{1,0}+\Phi^{1,0}\nabla^{0,1}_h=0$ is equivalent to the following four equations
\begin{equation}\label{prop 4.5 eq 3}
\begin{split}
&p^\ast\nabla^{0,1}_{h_1}p^\ast\Psi^{1,0}_1+p^\ast (\psi\circ\phi)\otimes q^\ast (\alpha\wedge\beta)+p^\ast\Psi^{1,0}_1p^\ast\nabla^{0,1}_{h_1}\\
=&p^\ast (\nabla^{0,1}_{h_1}\Psi^{1,0}_1+\Psi^{1,0}_1\nabla^{0,1}_{h_1})+p^\ast (\psi\circ\phi)\otimes q^\ast (\alpha\wedge\beta)\\
=&0,\\
&p^\ast(-\psi \Psi^{1,0}_2)\otimes q^\ast \alpha+p^\ast(\Psi^{1,0}_1 \psi)\otimes q^\ast \alpha\\
=&p^\ast (-\psi \Psi^{1,0}_2+\Psi^{1,0}_1 \psi)\otimes q^\ast \alpha\\
=&0,\\
&(p^\ast\nabla^{0,1}_{h_2}\otimes Id +Id\otimes q^\ast\nabla^{0,1}_{h^{(2)}})p^\ast \phi\otimes q^\ast\beta+(p^\ast \phi\otimes q^\ast\beta)p^\ast\nabla^{0,1}_{h_1}\\
=&p^\ast (\nabla^{0,1}_{h_2}\phi-\phi\nabla^{0,1}_{h_1})\otimes q^\ast \beta\\
=&0,\\
&(p^\ast\nabla^{0,1}_{h_2}\otimes Id +Id\otimes q^\ast\nabla^{0,1}_{h^{(2)}}) p^\ast\Psi^{1,0}_2+p^\ast (\phi\circ\psi)\otimes q^\ast(\beta\wedge \alpha) + p^\ast\Psi^{1,0}_2(p^\ast\nabla^{0,1}_{h_2}\otimes Id +Id\otimes q^\ast\nabla^{0,1}_{h^{(2)}})\\
=&p^\ast (\nabla^{0,1}_{h_2}\Psi^{1,0}_2+\Psi^{1,0}_2\nabla^{0,1}_{h_2})+p^\ast (\phi\circ\psi)\otimes q^\ast (\beta\wedge\alpha)\\
=&0.
\end{split}
\end{equation}

Since $p^\ast (\nabla^{0,1}_{h_i}\Psi^{1,0}_2+\Psi^{1,0}_2\nabla^{0,1}_{h_i})$ only contains the differential form of $X$ and $\alpha\wedge\beta$ is a non-degenerate (1,1)-form on $\mathbb{P}^1$, the first and fouth equation of (\ref{prop 4.5 eq 3}) is equivalent to 
\begin{align*}
&\nabla^{0,1}_{h_i}\Psi^{1,0}_i+\Psi^{1,0}_i\nabla^{0,1}_{h_i}=0\,\,(i=1,2),\\
&\phi\circ\psi=\psi\circ\phi=0.
\end{align*}
Recall that $\alpha$ and $\beta$ are everywhere non-zero. Hence, the second and the third equations of (\ref{prop 4.5 eq 3}) are equivalent to 
\begin{align*}
&-\psi \Psi^{1,0}_2+\Psi^{1,0}_1 \psi=0,\\
&\nabla^{0,1}_{h_2}\phi-\phi\nabla^{0,1}_{h_1}=0.
\end{align*}
Hence combining these equations , (\ref{prop 4.5 eq 1}), and  (\ref{prop 4.5 eq 2}), $(\nabla^{0,1}_h+\sqrt{-1}\Phi^{1,0})^2=0$ is equivalent to 
\begin{align*}
&(\nabla^{0,1}_{h_1}+\sqrt{-1}\Psi^{1,0}_1)^2=(\nabla^{0,1}_{h_2}+\sqrt{-1}\Psi^{1,0}_2)^2=0,\\
&\phi\circ(\nabla^{0,1}_{h_1}+\sqrt{-1}\Psi^{1,0}_1)=(\nabla^{0,1}_{h_2}+\sqrt{-1}\Psi^{1,0}_2)\circ\phi,\\
&(\nabla^{0,1}_{h_1}+\sqrt{-1}\Psi^{1,0}_1)\circ\psi=\psi\circ(\nabla^{0,1}_{h_2}+\sqrt{-1}\Psi^{1,0}_2),\\
&\phi\circ\psi=\psi\circ\phi=0.
\end{align*}
We proved the claim.
\end{proof}
Recall that in Section \ref{sec2}, we defined a subset $N\subset M$ as 
\begin{align*}
N=
\left\{
(\nabla_{h_1}, \Phi_1, \nabla_{h_2}, \Phi_2, \phi, \psi) \in M \ \middle| \
\begin{aligned}
&1.\quad (\nabla^{0,1}_i + \Phi^{1,0}_i)^2 = 0 \quad (i = 1,2) \\
&2.\quad \phi \circ (\nabla^{0,1}_1 + \Phi^{1,0}_1) = (\nabla^{0,1}_2 + \Phi^{1,0}_2) \circ \phi \\
&3.\quad (\nabla^{0,1}_1 + \Phi^{1,0}_1) \circ \psi = \psi \circ (\nabla^{0,1}_2 + \Phi^{1,0}_2) \\
&4.\quad \phi \circ \psi = \psi \circ \phi = 0
\end{aligned}
\right\}.
\end{align*}
Let $\iota:\mathcal{A}^{SU(2)}_{\mathrm{Conn},F}\to M$ be the map we defined in the previous section.
The following immediately follows from Proposition \ref{prop 4.5}.

\begin{proposition}\label{prop 4.6}
The restricition of $\iota_h$  to $\mathcal{A}^{SU(2)}_{\mathrm{Higgs},h}$ induces bijective maps
\begin{align*}
\iota_h&: \mathcal{A}^{SU(2)}_{\mathrm{Higgs},h,F}\to N,\\
\iota_h\circ\kappa&: \mathcal{A}^{SU(2)}_{\mathrm{Higgs},F}\to N.
\end{align*}

\end{proposition}

\subsection{Dimensional Reduction and the Doubly-Coupled Vortex Equation}

Recall that $(E_1,h_1)$ and $(E_2,h_2)$ are hermitian vector bundles over the Riemann surface $(X,\omega)$ of rank $\mathrm{rank}E_1=r_1, \mathrm{rank}E_2=r_2$.
The vector bundle $F$  is the bundle over $X\times \mathbb{P}^1$ defined as 

\begin{equation*}
F=F_1\oplus F_2=p^\ast E_1\oplus p^\ast E_2\otimes q^\ast\mathcal{O}_{\mathbb{P}^1}(2)
\end{equation*}
with the Hermitian metric

\begin{equation*}
h=h_{F_1}\oplus h_{F_2}=p^\ast h_1\oplus p^\ast h_2\otimes q^\ast h^{(2)}.
\end{equation*}

Here $h^{(2)}$ is the $SU(2)$-invariant metric on $\mathcal{O}_{\mathbb{P}^1}(2)$ which we construced in Section \ref{sec2}.\par
Let $D\in \mathcal{A}^{SU(2)}_{\mathrm{Higgs},h,F}$. 
In this section, we show that if $D$ is a Hermitian-Einstein connection for the metric $h$ and for certain K\"ahler forms on $X\times\mathbb{P}^1$, $\iota_h(D)\in N$ satisfies the doubly-coupled vortex equation and vice versa.\par
We now fix $\sigma\in\mathbb{R}_{>0}$.
Let 

\begin{equation}\label{sec 4.3 eq 1}
\tau:=\frac{\mathrm{deg}E_1+\mathrm{deg}E_2+\sigma\cdot r_2}{r_1+r_2}
\end{equation}

and let 
\begin{equation*}
\tau'=-\frac{1}{r_2}\cdot(r_1\cdot\tau-\mathrm{deg}E_1-\mathrm{deg}E_2)
\end{equation*}
\par
Under the above assumption, we define an $SU(2)$-invariant K\"ahler metric on $X\times\mathbb{P}^1$ whose K\"ahler form is 

\begin{equation*}
\Omega_\sigma=\bigg(\frac{\sigma}{2}p^\ast \omega\bigg)\oplus  q^\ast\omega_{\mathbb{P}^1}
\end{equation*}

where $\omega_{\mathbb{P}^1}$ is the Fubini-Study K\"ahler metric such that $\int_{\mathbb{P}^1}\omega_{\mathbb{P}^1}=1$. 
We also normalize $\omega$ so that $\int_X\omega=1$.

\begin{lemma}\label{lem 4.1}
Let $V_1, V_2$ be complex vector bundles over $X$ of $\mathrm{rank}\,V_2=r_2$.
Let $G$ be the complex vector bundle over $X\times \mathbb{P}^1$ defined as
\begin{equation*}
G=p^\ast V_1\oplus p^\ast V_2\otimes q^\ast \mathcal{O}_{\mathbb{P}^1}(2).
\end{equation*}
Then the degree of $G$ with respect to the K\"ahler form $\Omega_\sigma=(\frac{\sigma}{2} p^\ast\omega)\oplus q^\ast\omega_{\mathbb{P}^1}$ is 
\begin{equation*}
\mathrm{deg}G=\mathrm{deg}V_1+\mathrm{deg}V_2+\sigma r_2.
\end{equation*}
\begin{proof}
Recall that 
\begin{equation*}
\mathrm{deg}G=\int_{X\times \mathbb{P}^1}c_1(G)\wedge \Omega_\sigma.
\end{equation*}
We also have 
\begin{equation*}
c_1(G)=c_1(p^\ast V_1)+c_1(p^\ast V_2)+r_2\cdot c_1(q^\ast \mathcal{O}_{\mathbb{P}^1}(2)).
\end{equation*}
Since we assumed $\int_X\omega=\int_{\mathbb{P}^1}\omega_{\mathbb{P}^1}=1$, we have
\begin{align*}
\mathrm{deg}G&=\int_{X\times \mathbb{P}^1}c_1(G)\wedge \Omega_\sigma\\
&=\int_{X\times \mathbb{P}^1}(c_1(p^\ast V_1)+c_1(p^\ast V_2)+r_2\cdot c_1(q^\ast \mathcal{O}_{\mathbb{P}^1}(2)))\wedge \Omega_\sigma\\
&=\int_{X\times \mathbb{P}^1}c_1(p^\ast V_1)\wedge \Omega_\sigma+\int_{X\times \mathbb{P}^1}c_1(p^\ast V_2)\wedge \Omega_\sigma+\int_{X\times \mathbb{P}^1}r_2\cdot c_1(q^\ast \mathcal{O}_{\mathbb{P}^1}(2))\wedge \Omega_\sigma\\
&=\int_Xc_1(V_1)+\int_Xc_1(V_2)+r_2\cdot\frac{\sigma}{2}\int_{\mathbb{P}^1}c_1( \mathcal{O}_{\mathbb{P}^1}(2))\\
&=\mathrm{deg}V_1+\mathrm{deg}V_2+r_2\cdot\sigma.
\end{align*}
\end{proof}
\end{lemma}

\begin{proposition}\label{prop 4.7}
Let $D=\nabla_h+\sqrt{-1}\Phi\in \mathcal{A}^{SU(2)}_{\mathrm{Higgs},h,F}$.
Suppose $\sigma$ and $\tau$ are related as in (\ref{sec 4.3 eq 1}).
 Then the following conditions are equivalent.
\begin{itemize}
\item[1.] $(F,\nabla^{0,1}_h,\sqrt{-1}\Phi^{1,0})$ is polystable with respect to the $SU(2)$-action and the K\"ahler metric $\Omega_\sigma$.
\item[2.] $(F,\nabla^{0,1}_h,\sqrt{-1}\Phi^{1,0},h)$ is Hermitian-Einstein with respect to the K\"ahler metric $\Omega_\sigma$.
\item[3.] $\iota_h(D)=(\nabla_{h_1}, \Phi_1, \nabla_{h_2}, \Phi_2, \phi, \psi)$ is a solution of the doubly-coupled $\tau$-vortex equation.
\end{itemize}
\end{proposition}

\begin{proof}
The first two equivalences follow from the result of Simpson. We prove the equivalence of $2$ and $3.$\par
Let $D=\nabla_h+\sqrt{-1}\Phi\in \mathcal{A}^{SU(2)}_{\mathrm{Higgs},h,F}$. 
Let $F_D$ be the curvature of $D$.
Since we deal with the Hermitian-Einstein metric, we are only interested in $F^{(1,1)}_D$, which is the $(1,1)$-part of $F_D$.
We first calculate $F^{(1,1)}_D$.
Since $D\in \mathcal{A}^{SU(2)}_{\mathrm{Higgs},h}$, we have 
\begin{equation*}
F^{(1,1)}_D=F_{\nabla_h}-(\Phi\wedge\Phi)^{(1,1)}.
\end{equation*}
Here $F_{\nabla_h}$ is the curvature of $\nabla_h$ and $(\Phi\wedge\Phi)^{(1,1)}$ is the $(1,1)$-part of $\Phi\wedge\Phi$.
We first calculate $F_{\nabla_h}$.
\begin{align*}
&F_{\nabla_h}=(\nabla_h)^2\\
=&
\begin{pmatrix}
p^\ast\nabla_{h_1}& p^\ast \psi\otimes q^\ast \alpha \\
-(p^\ast \psi\otimes q^\ast \alpha)^\ast& p^\ast\nabla_{h_2}\otimes Id +Id\otimes q^\ast\nabla_{h^{(2)}}
\end{pmatrix}^2\\
=&\begin{pmatrix}
p^\ast F_{\nabla_{h_1}}-p^\ast \psi\otimes q^\ast \alpha\wedge(p^\ast \psi\otimes q^\ast \alpha)^\ast & \widetilde{\nabla}_{\mathrm{Hom}(F_2,F_1)}(p^\ast \psi\otimes q^\ast \alpha)\\
-\widetilde{\nabla}_{\mathrm{Hom}(F_1,F_2)}(p^\ast \psi\otimes q^\ast \alpha)^\ast& (p^\ast\nabla_{h_2}\otimes Id +Id\otimes q^\ast\nabla_{h^{(2)}})^2-(p^\ast \psi\otimes q^\ast \alpha)^\ast\wedge p^\ast \psi\otimes q^\ast \alpha
\end{pmatrix}\\
=&
\begin{pmatrix}
p^\ast F_{\nabla_{h_1}}-p^\ast \psi\otimes q^\ast \alpha\wedge(p^\ast \psi\otimes q^\ast \alpha)^\ast & \widetilde{\nabla}_{\mathrm{Hom}(F_2,F_1)}(p^\ast \psi\otimes q^\ast \alpha)\\
-\widetilde{\nabla}_{\mathrm{Hom}(F_1,F_2)}(p^\ast \psi\otimes q^\ast \alpha)^\ast& p^\ast F_{\nabla_{h_2}}+q^\ast(\overline\partial_{\mathbb{P}^1}((h^{(2)})^{-1}\partial_{\mathbb{P}^1}h^{(2)}))\otimes Id_{F_2}-(p^\ast \psi\otimes q^\ast \alpha)^\ast\wedge p^\ast \psi\otimes q^\ast \alpha
\end{pmatrix}.
\end{align*}
Here $F_{\nabla_{h_i}}$ are the curvatures of $\nabla_{h_i}$, $\widetilde{\nabla}_{\mathrm{Hom}(F_1,F_2)}$ (resp. $\widetilde{\nabla}_{\mathrm{Hom}(F_2,F_1)})$ is the connection of $\mathrm{Hom}(F_1,F_2)$ (resp. $\mathrm{Hom}(F_2,F_1)$) induced by the connections $\nabla_{h_i}$, and $\overline\partial_{\mathbb{P}^1}((h^{(2)})^{-1}\partial_{\mathbb{P}^1}h^{(2)})$ is the curvature of $\nabla_{h^{(2)}}$.\par
Recall that $(p^\ast \psi\otimes q^\ast \alpha)^\ast$ is the formal adjoint of $p^\ast \psi\otimes q^\ast \alpha$ with respect to the metric $h_{F_1}$ and $h_{F_2}$. Therefore 
\begin{equation*}
(p^\ast \psi\otimes q^\ast \alpha)^\ast=h^{-1}_{F_2} \overline{p^\ast\psi\otimes q^\ast \alpha}^Th_{F_1}.
\end{equation*}
Here $\overline{p^\ast\psi\otimes q^\ast \alpha}^T$ is the complex conjugate of $p^\ast\psi\otimes q^\ast \alpha$.
Since $h_{F_2}=p^\ast h_2\otimes q^\ast h^{(2)}$, after multipliying some constants to $\alpha$ we have 
\begin{align*}
p^\ast \psi\otimes q^\ast \alpha\wedge(p^\ast \psi\otimes q^\ast \alpha)^\ast &=\frac{2\sqrt{-1}}{\sigma} p^\ast (\psi\circ\psi^\ast)\otimes q^\ast\omega_{\mathbb{P}^1},\\
(p^\ast \psi\otimes q^\ast \alpha)^\ast\wedge p^\ast \psi\otimes q^\ast \alpha&=-\frac{2\sqrt{-1}}{\sigma} p^\ast (\psi^\ast\circ\psi)\otimes q^\ast\omega_{\mathbb{P}^1}.
\end{align*}
We next calculate $\Phi\wedge\Phi.$
\begin{align*}
&\Phi\wedge\Phi=
\begin{pmatrix}
p^\ast\Psi_1& -(p^\ast \phi\otimes q^\ast\beta)^\ast\\
p^\ast \phi\otimes q^\ast\beta& p^\ast\Psi_2
\end{pmatrix}^2\\
=&
\begin{pmatrix}
p^\ast(\Psi_1\wedge \Psi_1)-(p^\ast \phi\otimes q^\ast\beta)^\ast\wedge p^\ast \phi\otimes q^\ast\beta& -p^\ast\Psi_1\wedge(p^\ast \phi\otimes q^\ast\beta)^\ast-(p^\ast \phi\otimes q^\ast\beta)^\ast\wedge p^\ast\Psi_2\\
p^\ast \phi\otimes q^\ast\beta\wedge p^\ast\Psi_1+p^\ast\Psi_2 \wedge p^\ast \phi\otimes q^\ast\beta& -p^\ast \phi\otimes q^\ast\beta\wedge(p^\ast \phi\otimes q^\ast\beta)^\ast+ p^\ast(\Psi_2\wedge \Psi_2)
\end{pmatrix}.
\end{align*}
By the same argument as above, after multiplying some constants to $\beta$, we have 
\begin{align*}
(p^\ast \phi\otimes q^\ast\beta)^\ast\wedge p^\ast \phi\otimes q^\ast\beta&=\frac{2\sqrt{-1}}{\sigma} p^\ast (\phi^\ast\circ\phi)\otimes q^\ast\omega_{\mathbb{P}^1},\\
p^\ast \phi\otimes q^\ast\beta\wedge(p^\ast \phi\otimes q^\ast\beta)^\ast &=-\frac{2\sqrt{-1}}{\sigma} p^\ast (\phi\circ\phi^\ast)\otimes q^\ast\omega_{\mathbb{P}^1}.
\end{align*}
Suppose now $(F,\nabla^{0,1}_h,\sqrt{-1}\Phi^{1,0},h)$ is Hermitian-Einstein with respect to $\Omega_\sigma$. This is equivalent to 
\begin{align*}
\Lambda_{\sigma}(F_{\nabla_h}-\Phi\wedge\Phi)=\lambda Id_{F}.
\end{align*}
Here $\Lambda_\sigma$ is the contraction with respect to $\Omega_\sigma$.
Also by \cite[Lemma 2.1.8]{LT}, $\lambda$ is the constant such that 

\begin{equation*}
\lambda=-\frac{2\pi\sqrt{-1}}{\mathrm{Vol}_{\Omega_\sigma}(X\times \mathbb{P}^1)}\frac{\mathrm{deg}_\sigma F}{\mathrm{rank}F}.
\end{equation*}
Here $\mathrm{Vol}_{\Omega_\sigma}(X\times \mathbb{P}^1)$ and $\mathrm{deg}_\sigma F$ is the volume of $X\times \mathbb{P}^1$ and the degree of $F$ with respect to the K\"ahler form $\Omega_\sigma$.
The Hermitian-Einstein equation implies 
\begin{equation}\label{prop 4.7 eq 1}
\left\{\,
\begin{aligned}
&\Lambda_\sigma\bigg(p^\ast F_{\nabla_{h_1}}-p^\ast(\Psi_1\wedge \Psi_1)-\frac{2\sqrt{-1}}{\sigma}p^\ast (\phi^\ast\circ\phi)\otimes q^\ast\omega_{\mathbb{P}^1}-\frac{2\sqrt{-1}}{\sigma} p^\ast (\psi\circ\psi^\ast)\otimes q^\ast\omega_{\mathbb{P}^1}\bigg)=\lambda Id_{F_1}\\
&
\begin{split}
\Lambda_\sigma\bigg(p^\ast F_{\nabla_{h_2}}+q^\ast(\overline\partial_{\mathbb{P}^1}((h^{(2)})^{-1}&\partial_{\mathbb{P}^1}h^{(2)}))\otimes Id_{F_2}-p^\ast(\Psi_2\wedge \Psi_2)\\
&+\frac{2\sqrt{-1}}{\sigma} p^\ast (\phi\circ\phi^\ast)\otimes q^\ast\omega_{\mathbb{P}^1}+\frac{2\sqrt{-1}}{\sigma} p^\ast (\psi^\ast\circ\psi)\otimes q^\ast\omega_{\mathbb{P}^1}\bigg)=\lambda Id_{F_1}
\end{split}
\end{aligned}
\right.
\end{equation}
Since $\Omega_\sigma=(\frac{\sigma}{2}p^\ast\omega)\oplus\sigma q^\ast\omega_{\mathbb{P}^1}$, this equation is equivalent to
\begin{equation}\label{prop 4.7 eq 2}
\left\{\,
\begin{aligned}
&\frac{2}{\sigma}\Lambda\bigg( F_{\nabla_{h_1}}-\Psi_1\wedge \Psi_1\bigg)-\frac{2\sqrt{-1}}{\sigma}\phi^\ast\circ\phi-\frac{2\sqrt{-1}}{\sigma}\psi\circ\psi^\ast=\lambda Id_{E_1},\\
&\frac{2}{\sigma}\Lambda\bigg( F_{\nabla_{h_2}}-\Psi_2\wedge \Psi_2\bigg)+\frac{2\sqrt{-1}}{\sigma} \phi\circ\phi^\ast+\frac{2\sqrt{-1}}{\sigma}\psi^\ast\circ\psi-4\pi\sqrt{-1}\cdot Id_{E_2}=\lambda Id_{E_2}.
\end{aligned}
\right.
\end{equation}
To show that $\iota_h(D)=(\nabla_{h_1}, \Phi_1, \nabla_{h_2}, \Phi_2, \phi, \psi)$ is a solution of the doubly-coupled $\tau$-vortex equation, we need to verify the constants $\lambda, \tau$,  and $\tau'$ satisifies the relation

\begin{equation}\label{prop 4.7 eq 3}
\left\{\,
\begin{aligned}
&\frac{\sigma}{2}\cdot\lambda=-2\pi\sqrt{-1}\tau,\\
&\frac{\sigma}{2}\cdot(\lambda+4\pi\sqrt{-1})=-2\pi\sqrt{-1}\tau'.
\end{aligned}
\right.
\end{equation}

By Lemma \ref{lem 4.1}, we have 
\begin{align*}
\lambda&=-\frac{2\pi\sqrt{-1}}{\mathrm{Vol}_{\Omega_\sigma}(X\times \mathbb{P}^1)}\frac{\mathrm{deg}_\sigma F}{\mathrm{rank}F}\\
&=-\frac{4\pi\sqrt{-1}}{\sigma}\cdot\frac{\mathrm{deg}E_1+\mathrm{deg}E_2+\sigma\cdot r_2}{r_1+r_2}.
\end{align*}

Therefore, the equations (\ref{prop 4.7 eq 3}) hold.\par
We prove the converse direction.
Suppose that $\iota_h(D)=(\nabla_{h_1}, \Phi_1, \nabla_{h_2}, \Phi_2, \phi, \psi)$ is a solution of the doubly-coupled $\tau$-vortex equation.\par
Since (\ref{prop 4.7 eq 1}) and (\ref{prop 4.7 eq 2}) are equivalent and (\ref{prop 4.7 eq 2}) holds, to prove $(F,\nabla^{0,1}_h,\sqrt{-1}\Phi^{1,0},h)$ is Hermitian-Einstein, we only need to prove that the off-diagonal parts of $F_D$ are in the kernel of $\Lambda_\sigma$. This is equivalent to showing that   
\begin{align*}
&\Lambda_\sigma(\widetilde{\nabla}_{\mathrm{Hom}(F_2,F_1)}(p^\ast \psi\otimes q^\ast \alpha)-p^\ast\Psi_1\wedge(p^\ast \phi\otimes q^\ast\beta)^\ast-(p^\ast \phi\otimes q^\ast\beta)^\ast\wedge p^\ast\Psi_2)=0,\\
&\Lambda_\sigma(-\widetilde{\nabla}_{\mathrm{Hom}(F_1,F_2)}(p^\ast \psi\otimes q^\ast \alpha)^\ast+p^\ast \phi\otimes q^\ast\beta\wedge p^\ast\Psi_1+p^\ast\Psi_2 \wedge p^\ast \phi\otimes q^\ast\beta)=0
\end{align*}
holds.
These two equations hold because each differential form that appears in the above equation has a mixed contribution from $X$ and $\mathbb{P}^1$.
\end{proof} 

\section{Higgs Quadruplets and the Kobayashi-Hitchin Correspondence}\label{sec 5}
Let  $(X,\omega)$ be a compact Riemann surface, $\mathbb{P}^1$ be the projective line, $\omega_{\mathbb{P}^1}$ be the Fubini-Study form.
 We normalize $\omega$ and $\omega_{\mathbb{P}^1}$ so that $\int_X\omega=\int_{\mathbb{P}^1}\omega_{\mathbb{P}^1}=1$.
 We fix these spaces and K\"ahler forms throughout the section.  
We assume that $SU(2)$ acts on $X$ trivially and on $\mathbb{P}^1$ in the standard manner.  We freely use the notation that we used so far.\par
We would like to relate the existence of solutions of the doubly-coupled $\tau$-vortex equation to some holomorphic objects that satisfy a certain kind of stability.
We first introduce the holomorphic objects which we call \textit{Higgs quadruplet}.

\begin{definition}
We say a quadruplet $Q=((E_1, \overline\partial_{E_1},\theta_1),(E_2,\overline\partial_{E_2},\theta_2),\phi,\psi)$ is a Higgs quadruplet over $X$ if
\begin{itemize}
\item[1.] Each $(E_i, \overline\partial_{E_i},\theta_i)$ is a Higgs bundle over $X$ and 
\item[2.] $\phi:E_1\to E_2, \psi:E_2\to E_1$ are morphisms of Higgs bundles (i.e. $\phi\circ(\overline\partial_{E_1}+\theta_{E_1})=(\overline\partial_{E_2}+\theta_{E_2})\circ\phi, (\overline\partial_{E_1}+\theta_{E_1})\circ\psi=\psi\circ(\overline\partial_{E_2}+\theta_{E_2})$ holds).
\item[3.] $\psi\circ\phi=\phi\circ\psi=0$ holds.
\end{itemize}
\end{definition} 

Let  $Q=((E_1, \overline\partial_{E_1},\theta_1),(E_2,\overline\partial_{E_2},\theta_2),\phi,\psi)$ be a Higgs quadruplet over $X$.
Let $h_1$ and $h_2$ be hermitian metrics on $E_1$ and $E_2$,
$\partial_{h_i}$ be the $(1,0)$-part of the Chern connection with respect to $\overline\partial_{E_i}$ and $h_i$, and 
$\theta^\dagger_{h_i}$ be the formal adjoint of $\theta_i$ with respect to $h_i$.
Then, it is clear from the definition that  

\begin{equation*}
(\partial_{h_1}+\overline\partial_{E_1}, -\sqrt{-1}\theta_1+\sqrt{-1}\theta^\dagger_{h_1},\partial_{h_1}+\overline\partial_{E_2}, -\sqrt{-1}\theta_2+\sqrt{-1}\theta^\dagger_{h_2},-\sqrt{-1}\phi,\psi)\in N.
\end{equation*}

\begin{definition}
Let $Q=((E_1, \overline\partial_{E_1},\theta_1),(E_2,\overline\partial_{E_2},\theta_2),\phi,\psi)$ be a Higgs quadruplet over $X$.
Let $h_1$ and $h_2$ be hermitian metric of $E_1$ and $E_2$.
We say that $(Q,h_1,h_2)$ is a solution of the doubly-coupled $\tau$-vortex equation if 
\begin{equation*}
(\partial_{h_1}+\overline\partial_{E_1}, -\sqrt{-1}\theta_1+\sqrt{-1}\theta^\dagger_{h_1},\partial_{h_1}+\overline\partial_{E_2}, -\sqrt{-1}\theta_2+\sqrt{-1}\theta^\dagger_{h_2},-\sqrt{-1}\phi,\psi)\in N
\end{equation*}
is a solution of the doubly-coupled $\tau$-vortex equation.
This is equivalent to the following equations
\begin{equation*}
\left\{\,
\begin{split}
&\Lambda (F_{h_1}+[\theta_1,\theta^\dagger_{h_1}])-\sqrt{-1}\phi^\ast\circ \phi-\sqrt{-1}\psi\circ \psi^\ast+2\pi\sqrt{-1}\tau\mathrm{Id}_{E_1}=0,\\
&\Lambda (F_{h_2}+[\theta_2,\theta^\dagger_{h_2}])+\sqrt{-1}\phi\circ \phi^\ast+\sqrt{-1}\psi^\ast\circ \psi+2\pi\sqrt{-1}\tau'\mathrm{Id}_{E_2}=0.
\end{split}
\right.
\end{equation*}
\end{definition}

Let $F$ be an $SU(2)$-invariant vector bundle over $X\times \mathbb{P}^1$ defined as 
\begin{equation*}
F=p^\ast E_1\oplus p^\ast E_2\otimes q^\ast \mathcal{O}_{\mathbb{P}^1}(2).
\end{equation*}
Let $(F,\overline\partial_F,\theta_F)$ be the Higgs bundle defined as 
\begin{equation*}
\overline\partial_F:=
\begin{pmatrix}
p^\ast\overline\partial_{E_1}&p^\ast\psi\otimes q^\ast\alpha\\
0& p^\ast\overline\partial_{E_2}\otimes Id+Id\otimes q^\ast\overline\partial_{\mathcal{O}_{\mathbb{P}^1}(2)}
\end{pmatrix}
,\,\,
\theta_F:=
\begin{pmatrix}
p^\ast \theta_1&0\\
p^\ast\phi\otimes q^\ast \beta&p^\ast\theta_2.
\end{pmatrix}
\end{equation*}

$(F,\overline\partial_F,\theta_F)$ is indeed a Higgs bundle from Proposition \ref{prop 4.3}.
From Proposition \ref{prop 4.7}, we have the following.

\begin{proposition}\label{prop 5.1}
Let $Q=((E_1, \overline\partial_{E_1},\theta_1),(E_2,\overline\partial_{E_2},\theta_2),\phi,\psi)$ be a Higgs quadruplet over $X$. 
Let $\sigma\in\mathbb{R}_{>0}$ and 
\begin{equation*}
\tau=\frac{\mathrm{deg}E_1+\mathrm{deg}E_2+\sigma\cdot r_2}{r_1+r_2}.
\end{equation*}
We fix the K\"ahler form $\Omega_\sigma=(\frac{\sigma}{2}p^\ast\omega)\oplus q^\ast\omega_{\mathbb{P}^1}$ on $X\times \mathbb{P}^1$.
Then, the following two are equivalent.
\begin{itemize}
\item[1.] $E_1$ and $E_2$ admits hermitian metrics $h_1$ and $h_2$ such that   $(Q,h_1,h_2)$ is a solution of the doubly-coupled $\tau$-vortex equation.
\item[2.] $(F,\overline\partial_F,\theta_F)$ is polystable with respect to the $SU(2)$-action and the K\"ahler metric $\Omega_\sigma$.
\end{itemize}
Moreover, $h_1$ and $h_2$ are unique up to positive constants if $(F,\overline\partial_F,\theta_F)$ is stable with respect to the $SU(2)$-action.
\end{proposition}
\begin{proof}
Suppose $(Q,h_1,h_2)$ is a solution of the doubly-coupled $\tau$-vortex equation.
Then by Proposition \ref{prop 4.7}, $p^\ast h_1\oplus p^\ast h_2\otimes q^\ast h^{(2)}$ is an $SU(2)$-invariant Hermitian-Einstein metric on $F$ with respect to the K\"ahler form $\Omega_\sigma$.
Then by Theorem \ref{thm 2.1}, $F$ is polystable with respect to the $SU(2)$-action and $\Omega_\sigma$.\par
Conversely, let $F$ be polystable with respect to the $SU(2)$-action and $\Omega_\sigma$. 
Then again, by Theorem \ref{thm 2.1}, there exists an $SU(2)$-invariant Hermitian-Einstein metric $h$.
Then by Proposition \ref{prop 4.1}, $h$ has the form 
\begin{equation*}
h=p^\ast h_1\oplus p^\ast h_2\otimes q^\ast h^{(2)}.
\end{equation*}
By Proposition \ref{prop 4.7}, $(Q,h_1,h_2)$ is a solution of the doubly-coupled $\tau$-vortex equation.
If $F$ is stable, then the Hermitian-Einstein metric is unique up to positive constants, and hence the uniqueness of $h_1$ and $h_2$ follows.
\end{proof}

\subsection{Stability of Higgs Quadruplets}\label{sec 5.1}
Let $Q=((E_1, \overline\partial_{E_1},\theta_1),(E_2,\overline\partial_{E_2},\theta_2),\phi,\psi)$ be a Higgs quadruplet over $X$.
We would like to relate the existence of the solution of the doubly-coupled vortex equation to the stability of Higgs quadruplets.
Before we introduce the stability, we introduce the sub-objects and the sum of Higgs quadruplets.
\begin{definition}
Let $Q=((E_1, \overline\partial_{E_1},\theta_1),(E_2,\overline\partial_{E_2},\theta_2),\phi,\psi)$ be a Higgs quadruplet over $X$. 
A Higgs quadruplet $Q'=((E'_1, \overline\partial_{E'_1},\theta'_1),(E'_2,\overline\partial_{E'_2},\theta'_2),\phi',\psi')$ is sub-Higgs quadruplet of $Q$ if 
\begin{itemize}
\item $(E'_i, \overline\partial_{E'_i},\theta'_i)$ is a sub-Higgs bundle of $(E_i, \overline\partial_{E_i},\theta_i)$ for $i=1,2.$
\item The following diagrams are commutative.
\begin{equation*}
\begin{CD}
     (E_1, \overline\partial_{E_1},\theta_1) @>{\phi}>> (E_2, \overline\partial_{E_2},\theta_2) \\
  @AAA   @AAA \\
     (E'_1, \overline\partial_{E'_1},\theta'_1)  @>{\phi'}>>  (E'_2, \overline\partial_{E'_2},\theta'_2)
  \end{CD}
 ,\,\,\,\,\,
\begin{CD}
     (E_2, \overline\partial_{E_2},\theta_2) @>{\psi}>> (E_1, \overline\partial_{E_1},\theta_1) \\
  @AAA   @AAA \\
     (E'_2, \overline\partial_{E'_2},\theta'_2)  @>{\psi'}>>  (E'_1, \overline\partial_{E'_1},\theta'_1)
  \end{CD}.
  \end{equation*}
\end{itemize}
We call the zero Higgs quadruplet $Q'=0$ obtained by taking $E_1=E_2=0$ and the sub-Higgs quadruplet $Q'=Q$ trivial subquadruplets.
\end{definition}

\begin{definition}
Let $Q_E= ((E_1, \overline\partial_{E_1},\theta_{E_1}),(E_2,\overline\partial_{E_2},\theta_{E_2}),\phi,\psi)$ and $Q_F= ((F_1, \overline\partial_{F_1},\theta_1),(F_2,\overline\partial_{F_2},\theta_{F_2}),f,g)$ be Higgs quadruplets on $X$. 
We define 
\begin{align*}
&Q_E\oplus Q_F\\
:=&\bigg(\bigg(
E_1\oplus F_1, 
\begin{pmatrix}
\overline\partial_{E_1}&0\\
0&\overline\partial_{F_1}
\end{pmatrix}
,
\begin{pmatrix}
\theta_{E_1}&0\\
0&\theta_{F_1}
\end{pmatrix}
\bigg),
\bigg(
E_2\oplus F_2
\begin{pmatrix}
\overline\partial_{E_2}&0\\
0&\overline\partial_{F_2}
\end{pmatrix},
\begin{pmatrix}
\theta_{E_2}&0\\
0&\theta_{F_2}
\end{pmatrix}
\bigg),
\begin{pmatrix}
\phi&0\\
0&f
\end{pmatrix}
,
\begin{pmatrix}
\psi&0\\
0&g
\end{pmatrix}
\bigg)
\end{align*}
\end{definition}

We are now able to define stability for a given quadruplet following \cite{BG}.
 
\begin{definition}
Let $Q=((E_1, \overline\partial_{E_1},\theta_1),(E_2,\overline\partial_{E_2},\theta_2),\phi,\psi)$ be a Higgs quadruplet and
 $Q':=((E'_1, \overline\partial_{E'_1},\theta'_1),(E'_2,\overline\partial_{E'_2},\theta'_2),\phi',\psi')$ be a non-trivial sub-Higgs quadruplet of $Q$.
 Suppose $r_i:=\mathrm{rank}E_i,r'_i:=\mathrm{rank}E'_i$ for $i=1,2.$
 For any $\tau\in\mathbb{R}$, we define 
 \begin{equation*}
 \Theta_\tau(Q'):=(\mu(E'_1\oplus E'_2)-\tau)-\frac{r'_2}{r_2}\frac{r_1+r_2}{r'_1+r'_2}(\mu(E_1\oplus E_2)-\tau).
 \end{equation*}
We say that the quadruplet $Q$ is $\tau$-stable if 
 \begin{equation*}
 \Theta_\tau(Q')<0
 \end{equation*}
 for all nontrivial sub-Higgs quadruplets. 
 We say $Q$ is $\tau$-semistable if $ \Theta_\tau(Q')<0$ for all nontrivial sub-Higgs quadruplets.
 We say $Q$ is $\tau$-polystable if $Q=\oplus_i Q_i$ such that $Q_i$ is stable with $\Theta(Q)=\Theta(Q_i)$.
 \end{definition}
 
 We define another kind of stability.
 
 \begin{definition}
 Let $Q=((E_1, \overline\partial_{E_1},\theta_1),(E_2,\overline\partial_{E_2},\theta_2),\phi,\psi)$ be a Higgs quadruplet of $\mathrm{rank}E_1=r_1,\mathrm{rank}E_2=r_2.$
 Let $\sigma\in\mathbb{R}$.
 For $Q$, we define $\sigma$-degree $\mathrm{deg}_\sigma (Q)$ and $\sigma$-slope $\mu_\sigma(Q)$ as
\begin{equation*}
\mathrm{deg}_\sigma (Q)=\mathrm{deg}E_1+ \mathrm{deg}E_2+r_2\sigma
\end{equation*}
and 
\begin{equation*}
\mu_\sigma(Q):=\frac{\mathrm{deg}_\sigma (Q)}{r_1+r_2}.
\end{equation*}   
We say that the quadruplet $Q$ is $\sigma$-stable if
\begin{equation*}
\mu_\sigma(Q')<\mu_\sigma(Q)
\end{equation*} 
for all non-trivial sub-Higgs quadruplets.
We say that $Q$ is $\sigma$-semistable if $\mu_\sigma(Q')\leq\mu_\sigma(Q)$ for all non-trivial sub-Higgs quadruplets.
We say $Q$ is $\sigma$-polystable if $Q=\oplus_i Q_i$ such that $Q_i$ is stable with $\mu_\sigma(Q)=\mu_\sigma(Q_i)$.
\end{definition}

 These stabilities are equivalent once for an appropriate $\tau$ and $\sigma.$ 
 In particular, we have the following.
 
 \begin{proposition}[{\cite[Proposition 3.4]{G2}}]\label{prop 5.2}
 Let $Q=((E_1, \overline\partial_{E_1},\theta_1),(E_2,\overline\partial_{E_2},\theta_2),\phi,\psi)$ be a Higgs quadruplet. 
 Fix $\tau$ and $\sigma$ such that $\tau=\mu_\sigma(Q)$.
 Then $Q$ is $\tau$-(semi)stable if and only if $\sigma$-(semi)stable.
  \end{proposition}
\subsection{Sub-Higgs Quadruplets and Sub-Higgs bundles}
 Let $Q=((E_1, \overline\partial_{E_1},\theta_1),(E_2,\overline\partial_{E_2},\theta_2),\phi,\psi)$ be a Higgs quadruplet over $X$.
 Let $F$ be the $SU(2)$-equivalent vector bundle over $X\times \mathbb{P}^1$ defined as 
 \begin{equation*}
F=p^\ast E_1\oplus p^\ast E_2\otimes q^\ast \mathcal{O}_{\mathbb{P}^1}(2).
\end{equation*}
 Let $(F,\overline\partial_F,\theta_F)$ be the Higgs bundle defined as 
\begin{equation*}
\overline\partial_F:=
\begin{pmatrix}
p^\ast\overline\partial_{E_1}&p^\ast\psi\otimes q^\ast\alpha\\
0& p^\ast\overline\partial_{E_2}\otimes Id+Id\otimes q^\ast\overline\partial_{\mathcal{O}_{\mathbb{P}^1}(2)}
\end{pmatrix}
,\,\,
\theta_F:=
\begin{pmatrix}
p^\ast \theta_1&0\\
p^\ast\phi\otimes q^\ast \beta&p^\ast\theta_2
\end{pmatrix}.
\end{equation*}
The following will be used later.

\begin{lemma}\label{lem 5.1}
Let $x\in A^{1,0}(\mathrm{End}F)$ be an $SU(2)$-invariant section.
Then there exists $x_1\in A^{1,0}(\mathrm{End}E_1), x_1\in A^{1,0}(\mathrm{End}E_2),\phi\in A(\mathrm{Hom}(E_1,E_2))$ such that
\begin{equation*}
x=
\begin{pmatrix}
p^\ast x_1&0\\
p^\ast\phi\otimes q^\ast\beta&p^\ast x_2
\end{pmatrix}.
\end{equation*}
\end{lemma}
\begin{proof}
This follows from the argument in Proposition \ref{prop 4.3}.
\end{proof}
The goal of this subsection is to prove the following, which is the Higgs quadruplets analog of \cite[Lemma 4.3]{BG}.

\begin{proposition}\label{prop 5.3}
There exists a bijective correspondence between sub-Higgs quadruplets of $Q$ and $SU(2)$-invariant sub-Higgs bundles of $F$.
\end{proposition}
Before we give the proof, we first recall \cite[Lemma 4.3]{BG}.
In \cite{BG}, \textit{holomorphic triple} on $X$ was introduced.
Holomorphic triple on $X$ is a triple $T:=((V_1,\overline\partial_{V_1}),(V_2,\overline\partial_{V_2}),\gamma)$ such that each $(V_i,\overline\partial_{V_i})$ is a holomorphic bundle over $X$ and $\gamma:V_2\to V_1$ is a holomorphic bundle morphism (i.e. $ \overline\partial_{V_1}\circ\gamma=\gamma\circ\overline\partial_{V_2}$).\par
A sub-triple $T':=((V'_1,\overline\partial_{V'_1}),(V'_2,\overline\partial_{V'_2}),\gamma')$ of $T$ is a triple such that each $(V'_i,\overline\partial_{V'_i})$ is a sub-holomoprhic bundle of $(V_i,\overline\partial_{V_i})$ and the diagram 
\begin{equation*}
\begin{CD}
     (V_2, \overline\partial_{V_2}) @>{\gamma}>> (V_1, \overline\partial_{V_1}) \\
  @AAA   @AAA \\
     (V'_2, \overline\partial_{V'_2})  @>{\gamma'}>>  (V'_1, \overline\partial_{V'_1})
  \end{CD}
\end{equation*}
is commutative.\par
Let $G$ be a vector bundle over $X\times \mathbb{P}^1$ defined as $G=p^\ast V_1\oplus p^\ast V_2\otimes q^\ast\mathcal{O}_{\mathbb{P}^1}(2)$ with the holomorphic structure 
\begin{equation*}
\overline\partial_G
:=
\begin{pmatrix}
p^\ast \overline\partial_{V_1}& p^\ast\gamma\otimes q^\ast \alpha\\
0& p^\ast\overline\partial_{V_2}\otimes Id+Id\otimes q^\ast\overline\partial_{\mathcal{O}_{\mathbb{P}^1}(2)}
\end{pmatrix}.
\end{equation*}
$(G,\overline\partial_G)$ is indeed an $SU(2)$-invariant holomorphic bundle. 
See \cite{G2} or set $\Psi_1=\Psi_2=\phi=0$ in Proposition \ref{prop 4.5}.
\begin{lemma}[{\cite[Lemma 4.3]{BG}}]\label{lem 5.2}
Let $T$ and $G$ be as above. 
There exists a bijective correspondence between sub-triples of $T$ and $SU(2)$-invariant sub-holmorphic bundles of $G.$
In paricular, for every $SU(2)$-invariant sub-holomorphic bundles $G'\subset G$ there exists a subtriple $T':=((V'_1,\overline\partial_{V'_1}),(V'_2,\overline\partial_{V'_2}),\gamma')$ of $T$ such that $G'$ has the form
\begin{equation*}
G'=p^\ast V'_1\oplus p^\ast V'_2\otimes q^\ast\mathcal{O}_{\mathbb{P}^1}(2)
\end{equation*}
with the holomorphic structure
\begin{equation*}
\overline\partial_{G'}
:=
\begin{pmatrix}
p^\ast \overline\partial_{V'_1}& p^\ast\gamma'\otimes q^\ast \alpha\\
0& p^\ast\overline\partial_{V'_2}\otimes Id+Id\otimes q^\ast\overline\partial_{\mathcal{O}_{\mathbb{P}^1}(2)}
\end{pmatrix}.
\end{equation*}
\end{lemma}

We now prove Proposition \ref{prop 5.3}.
\begin{proof}[proof of Proposition \ref{prop 5.3}]
Let $Q':=((E'_1, \overline\partial_{E'_1},\theta'_1),(E'_2,\overline\partial_{E'_2},\theta'_2),\phi',\psi')$ be sub-Higgs quadruplet of $Q$.

Then we set 
 \begin{align*}
&F':=p^\ast E'_1\oplus p^\ast E'_2\otimes q^\ast \mathcal{O}_{\mathbb{P}^1}(2),\\
&\overline\partial_{F'}:=
\begin{pmatrix}
p^\ast\overline\partial_{E'_1}&p^\ast\psi'\otimes q^\ast\alpha\\
0& p^\ast\overline\partial_{E'_2}\otimes Id+Id\otimes q^\ast\overline\partial_{\mathcal{O}_{\mathbb{P}^1}(2)}
\end{pmatrix},\\
&\theta_{F'}:=
\begin{pmatrix}
p^\ast \theta'_1&0\\
p^\ast\phi'\otimes q^\ast \beta&p^\ast\theta'_2.
\end{pmatrix}.
\end{align*}
Then $(F',\overline\partial_{F'}, \theta_{F'})$ is an $SU(2)$-invariant Higgs bundle by Proposition \ref{prop 4.5}.
By the definition of sub-Higgs quadruplet, $\overline\partial_F|_{F'}=\overline\partial_{F'}$ and $\theta_F|_{F'}=\theta_{F'}$ holds. 
Hence $(F',\overline\partial_{F'}, \theta_{F'})$ is a sub-Higgs bundle of $F$.\par
Suppose $(F',\overline\partial_{F'}, \theta_{F'})$ is an $SU(2)$-invariant sub-Higgs bundle of $(F,\overline\partial_F,\theta_F)$.
Note that from the Higgs quadruplet $Q=((E_1, \overline\partial_{E_1},\theta_1),(E_2,\overline\partial_{E_2},\theta_2),\phi,\psi)$ we obtain a holomorphic triple 
\begin{equation*}
T_Q:=((E_1, \overline\partial_{E_1}),(E_2,\overline\partial_{E_2}), \psi).
\end{equation*}
Since $(F',\overline\partial_{F'})$ is an $SU(2)$-invariant sub-holomoprhic bundle of $(F,\overline\partial_F)$, we can apply Lemma \ref{lem 5.2} and obtain a sub-triple 
\begin{equation*}
T'_Q:=((E'_1, \overline\partial_{E'_1}),(E'_2,\overline\partial_{E'_2}), \psi')
\end{equation*}
of $T_Q$ such that 
 \begin{align*}
&F'=p^\ast E'_1\oplus p^\ast E'_2\otimes q^\ast \mathcal{O}_{\mathbb{P}^1}(2),\\
&\overline\partial_{F'}=
\begin{pmatrix}
p^\ast\overline\partial_{E'_1}&p^\ast\psi'\otimes q^\ast\alpha\\
0& p^\ast\overline\partial_{E'_2}\otimes Id+Id\otimes q^\ast\overline\partial_{\mathcal{O}_{\mathbb{P}^1}(2)}
\end{pmatrix}
\end{align*}
holds. \par
We next study the structure of $\theta_{F'}$.
Recall that $\theta_{F'}\in A^{1,0}(\mathrm{End}F')$ and $\theta_{F'}$ is $SU(2)$-invariant.
Then by Lemma \ref{lem 5.1}, there exists $\theta_{E'_1}\in A^{1,0}(\mathrm{End}E'_1),
 \theta_{E'_2}\in A^{1,0}(\mathrm{End}E'_2)$, and $\phi'\in A(\mathrm{Hom}(E_1,E_2))$ such that 
\begin{equation*}
\theta_{F'}
=
\begin{pmatrix}
p^\ast \theta_{E'_1}&0\\
p^\ast\phi'\otimes q^\ast \beta&p^\ast \theta_{E'_2}
\end{pmatrix}.
\end{equation*}
Since $\theta_{F'}$ is a Higgs field, $(\overline\partial_{F'}+\theta_{F'})^2=0$ holds and hence by the calculation in Proposition \ref{prop 4.5} shows that
\begin{equation*}
Q':=((E'_1, \overline\partial_{E'_1},\theta'_1),(E'_2,\overline\partial_{E'_2},\theta'_2),\phi',\psi')
\end{equation*}
is a Higgs quadruplet.\par
Since $\theta_F|_{F'}=\theta_{F'}$, $Q'$ is a sub-Higgs quadruplet of $Q$.
The construction gives a bijective correspondence, and therefore, we finish the proof.
\end{proof}
\subsection{Kobayashi-Hitchin Correspondence}
 Let $Q=((E_1, \overline\partial_{E_1},\theta_1),(E_2,\overline\partial_{E_2},\theta_2),\phi,\psi)$ be a Higgs quadruplet over $X$.
 Let $F$ be the $SU(2)$-equivalent vector bundle over $X\times \mathbb{P}^1$ defined as 
 \begin{equation*}
F=p^\ast E_1\oplus p^\ast E_2\otimes q^\ast \mathcal{O}_{\mathbb{P}^1}(2).
\end{equation*}
 Let $(F,\overline\partial_F,\theta_F)$ be the Higgs bundle defined as 
\begin{equation*}
\overline\partial_F:=
\begin{pmatrix}
p^\ast\overline\partial_{E_1}&p^\ast\psi\otimes q^\ast\alpha\\
0& p^\ast\overline\partial_{E_2}\otimes Id+Id\otimes q^\ast\overline\partial_{\mathcal{O}_{\mathbb{P}^1}(2)}
\end{pmatrix}
,\,\,
\theta_F:=
\begin{pmatrix}
p^\ast \theta_1&0\\
p^\ast\phi\otimes q^\ast \beta&p^\ast\theta_2.
\end{pmatrix}
\end{equation*}

\begin{lemma}\label{lem 5.3}
Let $E$ be a holomorphic vector bundle over $X\times \mathbb{P}^1$.
Every saturated sub-$\mathcal{O}_{X\times \mathbb{P}^1}$-module $\mathcal{V}\subset E$ is a locally free sheaf.
\end{lemma}
\begin{proof}
Recall that a saturated subsheaf $\mathcal{V}\subset E$ is a subsheaf of $E$ such that $E/\mathcal{V}$ is torsion-free.
Such sheaves are locally free on complex surfaces (See \cite[Corollary 5.5.20, Proposition 5.5.22]{Ko}).
\end{proof}

\begin{lemma}\label{lem 5.4}
Let $Q'=((E'_1, \overline\partial_{E'_1},\theta'_1),(E'_2,\overline\partial_{E'_2},\theta'_2),\phi',\psi')$ be any sub-Higgs quadruplet of $Q$ and let $F'$ be the sub Higgs bundle of $F$ which corresponds to $Q'$.
Then 
\begin{equation*}
\mu_\sigma(Q')=\mu(F').
\end{equation*}
Here, we calculate the degree of $F'$ by $\Omega_\sigma$.
\end{lemma}
\begin{proof}
Recall that $F'$ has the form
\begin{equation*}
F'=p^\ast E_1\otimes p^\ast E_2\otimes q^\ast\mathcal{O}_{\mathbb{P}^1}(2).
\end{equation*}
Then by Lemma \ref{lem 4.1} the degree of $F'$ with respect to the K\"ahler form $\Omega_\sigma$ is 
\begin{equation*}
\mathrm{deg}F'=\mathrm{deg}E'_1+\mathrm{deg}E'_2+\sigma \cdot\mathrm{rank}E'_2.
\end{equation*}
Hence 
\begin{equation*}
\mu_\sigma(Q')=\mu(F').
\end{equation*}
holds.
\end{proof}

\begin{proposition}\label{prop 5.4}
Let $\tau=\mu_\sigma(Q).$
The following three conditions are equivalent.
\begin{itemize}
\item[1.] $Q$ is $\sigma$-stable.
\item[2.] $Q$ is $\tau$-stable.
\item[3.] $F$ is stable with respect to the $SU(2)$-action and the K\"ahler form $\Omega_\sigma$.
\end{itemize}
\end{proposition}
\begin{proof}
The first two conditions are equivalent from \cite[Proposition 3.4]{BG}. We prove the equivalence of $1$ and $3$.\par
Let $Q'$ be any sub-Higgs quadruplet of $Q$ and let $F'$ be the corresponding $SU(2)$-invariant sub-Higgs bundle of $F$.
Then, by Proposition \ref{prop 5.3} and Lemma \ref{lem 5.4}, we have
\begin{equation*}
\mu_\sigma(Q')=\mu(F').
\end{equation*}
Therefore $Q$ is $\sigma$-stable if $F$ is stable with respect to the $SU(2)$-action.\par
The converse follows from Lemma \ref{lem 5.3} that all the saturated sub-Higgs sheaves of $F$ are actually a sub-Higgs bundle.
\end{proof}

We are now able to state the main result of the paper.

\begin{theorem}\label{thm 5.1}
Let $\sigma\in\mathbb{R}_{>0}$.
Let $Q=((E_1, \overline\partial_{E_1},\theta_1),(E_2,\overline\partial_{E_2},\theta_2),\phi,\psi)$ be a $\sigma$-stable Higgs quadruplet over $X$ and $\tau=\mu_\sigma(Q)$. 
Then there exist hermitian metrics $h_1$ and $h_2$ on $E_1$ and $E_2$ such that $(Q,h_1,h_2)$ is a solution of the doubly-coupled $\tau$-vortex equation.
The metrics $h_1$ and $h_2$ are unique up to positive constants.
\end{theorem}
\begin{proof}
If $Q$ is $\sigma$-stable, then $F$ is stable with respect to the $SU(2)$-action and the K\"ahler form $\Omega_\sigma$ by Proposition \ref{prop 5.4}.
Then by Theorem \ref{thm 2.1}, there exists an $SU(2)$-invariant Hermitian-Einstein metric $h$ on $F$.
It follows from Proposition \ref{prop 5.1} that there exist hermitian metrics $h_1$ and $h_2$ on $E_1$ and $E_2$ such that $(Q,h_1,h_2)$ is a solution of the doubly-coupled $\tau$-vortex equation.
\end{proof}

We also have 
\begin{proposition}\label{prop 5.5}
Let $Q=((E_1, \overline\partial_{E_1},\theta_1),(E_2,\overline\partial_{E_2},\theta_2),\phi,\psi)$ be a Higgs quadruplet of $\mathrm{rank} E_1=r_1$ and $\mathrm{rank} E_2=r_2$.
Let $(F,\overline\partial_F,\theta_F)$ be the $SU(2)$-invariant Higgs bundle whcih corresponds to $Q$.
Let $\tau\in\mathbb{R}$ be a constant such that 
\begin{equation*}
\sigma:=\frac{(r_1+r_2)\cdot\tau-\mathrm{deg}E_1-\mathrm{deg}E_2}{r_2}>0
\end{equation*}
holds.\par
If $E_1$ and $E_2$ admits hermitian metrics $h_1$ and $h_2$ such that $(Q,h_1,h_2)$ is a solution of the doubly-coupled $\tau$-vortex equation, then $Q$ is $\sigma$-stable.
\end{proposition}
\begin{proof}
Suppose that $E_1$ and $E_2$ admits hermitian metrics $h_1$ and $h_2$ such that $(Q,h_1,h_2)$ is a solution of the doubly-coupled $\tau$-vortex equation.
Then by Proposition \ref{prop 5.1}, $(F,\overline\partial_F,\theta_F)$ is polystable with respect to the $SU(2)$-action.
Hence there exists a decomposition $(F,\overline\partial_F,\theta_F)=\oplus_i(F_i,\overline\partial_{F_i},\theta_{F_i})$ such that  each $(F_i,\overline\partial_{F_i},\theta_{F_i})$ is stable and $\mu(F)=\mu(F_i)$.
This decomposition induces a decomposition $Q=\oplus_i Q_i$.
Each $Q_i$ and $(F_i,\overline\partial_{F_i},\theta_{F_i})$ are related as in Proposition \ref{prop 5.3}.
Hence each $Q_i$ is $\sigma$-stable by Proposition \ref{prop 5.4}.
Since $\mu_\sigma(Q)=\mu(F)$ from Lemma \ref{lem 5.4}, we conclude that $Q$ is $\sigma$-polystable.
\end{proof}

\end{document}